\documentclass[10pt]{amsart}

\usepackage{fourier}
\usepackage[margin=1in]{geometry}
\usepackage{amsmath, amsthm, amssymb}
\usepackage[colorlinks=true]{hyperref}
\hypersetup{urlcolor=blue, citecolor=red}

\newtheorem{Theorem}{Theorem}[section]

\newtheorem{Lemma}[Theorem]{Lemma}
\newtheorem{Proposition}[Theorem]{Proposition}
\newtheorem{Remark}[Theorem]{Remark}
\newtheorem{Assumption}[Theorem]{Assumption}

\def\C {\mathbb C}
\def\R {\mathbb R}

\newcommand{\supp}{\operatorname{supp}}

\newcommand{\id}{\operatorname{Id}}
\newcommand{\p}{\partial}
\renewcommand{\Re}{\operatorname{Re}}
\renewcommand{\Im}{\operatorname{Im}}
\newcommand{\Span}{\operatorname{Span}}
\newcommand{\curl}{\operatorname{curl}}
\newcommand{\Div}{\operatorname{Div}}
\newcommand{\bt}{\operatorname{\mathbf t}}
\newcommand{\Spec}{\operatorname{Spec}}

\title{Unique recovery of electrical conductivity and magnetic
  permeability from Magneto-Telluric data}

\author{Yernat M. Assylbekov}
\address{Department of Computational and Applied Mathematics, Rice University, Houston, TX 77005, USA}
\email{yernat.assylbekov@gmail.com}
\author{Maarten V. de Hoop}
\address{Department of Computational and Applied Mathematics, Rice University, Houston, TX 77005, USA}
\email{mdehoop@rice.edu}

\begin{document}
\maketitle

\begin{abstract}
We present a comprehensive mathematical study of the Magneto-Telluric
(MT) method, on bounded domain in $\mathbb{R}^3$. We show that
electrical conductivity and magnetic permeability, assumed to be
$C^2$, can be uniquely recovered from MT data measured on the boundary
of the domain. The proof is based on the construction of complex
geometric optics solutions. Furthermore, we obtain a unique
determination result in the case when the MT data are measured only on
an open subset of the boundary. Here, we assume that the part of the
boundary inaccessible for measurements is a subset of a sphere.
\end{abstract}

\section{Introduction}\label{section::introduction}

The magnetotelluric (MT) method uses electromagnetic passive sources
in the magnetosphere and ionosphere to estimate electrical resistivity
in Earth's interior \cite{chave1991electrical}. This method was
introduced almost 70 years ago by Cagniard
\cite{cagniard1953basic}. The passive sources excite a certain
spectrum, that is, interval of frequencies. There is a long history of
studies based on some form of data fitting including a Bayesian
approach \cite{grandis1999bayesian}. It has been understood that high
frequency data enable determining the resistivity in the ``near''
subsurface, while low frequency ones enable obtaining an estimate of
the resistivity in the ``deep'' interior. The MT method has been used
to study melt and hydration in Earth's crust and mantle
\cite{rokityansky2012geoelectromagnetic, feucht2017magnetotelluric,
  evans1994electrical, rippe2013magnetotelluric}, and general
electrical structure in the lithosphere
\cite{malleswari2019magnetotelluric, thiel2008modelling} and below the
ocean floor \cite{cox1971electromagnetic}. Moreover, the MT method has
been applied to earthquake prediction \cite{chouliaras1988application}.

\smallskip

In this paper, we analyze the inverse problem for the MT method. We
let $\Omega \subset \R^3$ be an open bounded domain with $C^{1,1}$
boundary and let $\varepsilon, \sigma, \mu \in C^2(\overline\Omega)$
be non-negative. We consider the time-harmonic Maxwell equations for
electromagnetic fields, $E$ and $H$,
\begin{equation}\label{eqn::Maxwell}
   \nabla \times E = i \omega \mu H\quad\text{and}\quad
   \nabla \times H = -i \omega
          \Big(\varepsilon + \frac{i \sigma}{\omega}\Big) E
   \quad\text{in}\quad\Omega ,
\end{equation}
where $\omega > 0$ is a fixed frequency. The functions $\varepsilon$,
$\sigma$ and $\mu$ represent the material parameters, namely, electrical
permittivity, electrical conductivity and magnetic permeability,
respectively. In the case of MT, one invokes the following

\begin{Assumption} \label{MT assumption}
The electrical permittivity vanishes, $\varepsilon = 0$, and the
electrical conductivity and magnetic permeability satisfy $\sigma \ge
\sigma_0$, $\mu \ge \mu_0$ on $\overline\Omega$ for some constants
$\sigma_0, \mu_0 > 0$.
\end{Assumption}

\noindent
This assumption has its origin in the low-frequency reduction of
Maxwell's equations.

We let $\nu$ be the outer unit normal to the boundary $\p \Omega$, and
define the trace operator $\bt: C^\infty(\overline\Omega;\C^3) \to
C^\infty(\p\Omega;\C^3)$ as
$$
   \bt(u) := \nu \times u|_{\p\Omega}\quad\text{for}\quad
   u \in C^\infty(\overline\Omega;\C^3) .
$$
The trace $\bt$ can be extended to a bounded linear operator from
$H_{\Div}^1(\Omega)$ into $TH^{1/2}_{\Div}(\p\Omega)$, where
$$
   H^1_{\Div}(\Omega) := \left\{
   u \in H^1(\Omega;\C^3) : \bt(u) \in TH^{1/2}_{\Div}(\p\Omega)\right\}
$$
and
$$
   TH^{1/2}_{\Div}(\p\Omega) := \left\{f \in H^{1/2}(\p \Omega;\C^3) :
   \nu \cdot f = 0\quad\text{and}\quad
   \Div(f)\in H^{1/2}(\p \Omega;\C^3)\right\} ,
$$
and $\Div$ denotes the surface divergence on $\p\Omega$. We refer the
reader to \cite{ola1993inverse} for more details. Under
Assumption~\ref{MT assumption}, when $\omega > 0$ does not belong to a
discrete set of magnetic resonant frequencies, the equation
\eqref{eqn::Maxwell} with the boundary condition $\bt(H) = f \in
TH^{1/2}_{\Div}(\p\Omega)$ has a unique solution $(H,E) \in
H^1_{\Div}(\Omega)\times H^1_{\Div}(\Omega)$; see
Section~\ref{section::well posedness}. Then the impedance map
$Z^\omega_{\sigma,\mu}: TH^{1/2}_{\Div}(\p\Omega)\to
TH^{1/2}_{\Div}(\p\Omega)$ is defined as
$$
   Z^\omega_{\sigma,\mu}(f)
           := \bt(E),\quad f\in TH^{1/2}_{\Div}(\p\Omega) .
$$
The inverse MT problem (IMTP) is to determine $\sigma$ and $\mu$ from
the knowledge $Z^\omega_{\sigma,\mu}$. In the geophysics literature,
one commonly considers the resistivity, that is, the reciprocal of
conductivity \cite{weckmann2003images}.

\smallskip

\begin{Remark}
In practice, the data are represented as the graph of
$Z^\omega_{\sigma,\mu} : TH^{1/2}_{\Div}(\p\Omega) \to
TH^{1/2}_{\Div}(\p\Omega)$.
\end{Remark}

Our first main result is

\begin{Theorem}\label{main thm}
Let $\Omega\subset \R^3$ be a bounded domain with $C^{1,1}$ boundary
and let $\sigma_j,\mu_j\in C^2(\overline\Omega)$, $j=1,2$, be such
that $\sigma_j \ge \sigma_0$ and $\mu_j \ge \mu_0$ for some constants
$\sigma_0, \mu_0 > 0$. Suppose that $\omega > 0$ is not a resonant
frequency for $(\sigma_1,\mu_1)$ and $(\sigma_2,\mu_2)$ and that
\begin{equation}\label{boundary assumption}
   \p^\alpha \sigma_1|_{\p\Omega} = \p^\alpha \sigma_2|_{\p\Omega}
   \quad\text{and}\quad
   \p^\alpha \mu_1|_{\p\Omega} = \p^\alpha \mu_2|_{\p\Omega}
   \quad\text{for}\quad|\alpha| \le 2 .
\end{equation}
Then $Z_{\sigma_1,\mu_1}^\omega = Z_{\sigma_2,\mu_2}^\omega$ implies
that $\sigma_1 = \sigma_2$ and $\mu_1 = \mu_2$.
\end{Theorem}

\noindent
Condition \eqref{boundary assumption} is not important. We expect that
a suitable boundary determination result would allow to remove it as
in \cite{caro2019boundary, joshi2000total, mcdowall1997boundary}. See
also Remark~\ref{remark in Appendix B}.

\smallskip

The inverse problem considered in the present paper formally looks
like a standard inverse electromagnetic problem (IEMP) proposed in
\cite{somersalo1992linearized}: Determine $\varepsilon$, $\mu$ and
$\sigma$, from the knowledge of the admittance map
$\Lambda^\omega_{\varepsilon,\sigma,\mu} : \bt(E) \mapsto \bt(H)$ for
all $(E,H) \in H_{\Div}^1(\Omega) \times H_{\Div}^1(\Omega)$ solving
\eqref{eqn::Maxwell}. However, the conditions in the IEMP do not allow
the vanishing of $\varepsilon$ as in our Assumption~\ref{MT
  assumption}. To be precise, IEMP invokes

\begin{Assumption} \label{EM assumption}
The electrical permittivity, electrical conductivity and magnetic
permeability satisfy $\varepsilon \ge \varepsilon_0$, $\mu \ge \mu_0$
and $\sigma \ge 0$ on $\overline\Omega$ for some constants
$\varepsilon_0, \mu_0 > 0$.
\end{Assumption}

\noindent
As a consequence, the IMTP and IEMP problems are in fact
different. For comparison, we also mention the EIT, or inverse
conductivity problem, also known as electrical resistivity tomography
(ERT) in the geophysics literature, proposed in
\cite{calderon2006inverse}: Determine $\sigma$, satisfying $\sigma \ge
\sigma_0$ for some constant $\sigma_0 > 0$, from the
Dirichlet-to-Neumann map $u|_{\p\Omega} \mapsto \sigma\p_\nu
u|_{\p\Omega}$ for all $u \in H^1(\Omega)$ solving the conductivity
equation $\nabla \cdot (\sigma \nabla u) = 0$ in $\p\Omega$. A
low-frequency limit of IEMP as in \cite{lassas1997impedance} will not
converge to EIT. However, one expects that a low-frequency limit of
IMTP meaningfully relates to EIT.

\smallskip

We now give a brief overview of results pertaining to the IEMP and its
history. The standard approach to solve this problem is to
construct a family of exponentially growing solutions, also known as
\emph{complex geometric optics solutions}, following the celebrated
paper~\cite{sylvester1987global} on the inverse conductivity
problem. One of the main challenges in adopting the method of
\cite{sylvester1987global} is the fact that \eqref{eqn::Maxwell} with
Assumption~\ref{EM assumption} is not elliptic. The linearized problem
at constant material parameters was studied in
\cite{somersalo1992linearized}. For the nonlinear problem, a
uniqueness result was given in \cite{sun1992inverse} when the
electromagnetic parameters are close to constants. In this paper, to
get ellipticity, equation \eqref{eqn::Maxwell} with Assumption~\ref{EM
  assumption} was reduced to a system whose principal part is the
Laplacian. However, this reduction gives some first order terms. For
material parameters that are nearly constant, the authors were able to
manage the first-order terms and introduce complex geometrical solutions
for \eqref{eqn::Maxwell}. The first global uniqueness result was
proven in~\cite{ola1993inverse}. This proof was later simplified
in~\cite{ola1996electromagnetic}. The important point in the
simplified proof is to augment \eqref{eqn::Maxwell} with
Assumption~\ref{EM assumption} to a certain $8 \times 8$ Dirac
equation and connect it via some other Dirac operator to an $8 \times
8$ system whose principal part is the Laplacian while its remainder
involves only zeroth-order terms. This allowed the authors to
construct complex geometric optics solutions for the latter system and
connect them to \eqref{eqn::Maxwell} with Assumption~\ref{EM
  assumption} by applying the Dirac operator that was initially
introduced. This technique became popular in the subsequent works on
various aspects of IEMP~\cite{assylbekov2017kerrinverse,
  caro2014global, kenig2011inverse}.

\smallskip

In the setting of the IMTP, however, one cannot simply employ the
complex geometric optics solutions constructed in
\cite{ola1993inverse, ola1996electromagnetic} for the IEMP. Moreover,
the elliptization argument of \cite{ola1996electromagnetic}, applied
to \eqref{eqn::Maxwell} with Assumption~\ref{MT assumption}, does not
help avoiding first-order terms. Instead of that, we
follow~\cite{sun1992inverse} and reduce \eqref{eqn::Maxwell} with
Assumption~\ref{MT assumption} to a system whose principal part is the
Laplacian. We then introduce novel complex geometrical optics
solutions for the reduced system that are essentially solutions for
\eqref{eqn::Maxwell} with Assumption~\ref{MT assumption}. Moreover,
using this reduction gives an integral identity with a clear relation
to \eqref{eqn::Maxwell}. To deal with the first-order terms, we use
the ideas from \cite{colton1992uniqueness} with substantial
modifications since the latter paper assumes that $\mu$ is constant.

\smallskip

In the MT method, performing measurements on the entire boundary (that
is, the surface of the earth) is impossible. Therefore, the analysis
of the inverse problem with local measurements is important. We can
assume that the measurements are performed on a nonempty open subset
$\Gamma$ of $\p\Omega$ only and that the inaccessible part of the
boundary $\Gamma_0 = \overline{\p\Omega \setminus \Gamma}$ is a part
of a sphere (our planet's surface). Our second main result is the
following

\begin{Theorem}\label{main thm 2}
Let $\Omega \subset B_0$ be a bounded domain with $C^{1,1}$ boundary
included in an open ball $B_0 \subset \R^3$ and let $\Gamma_0 =
\p\Omega \cap \p B_0$, $\Gamma_0 \neq \p B_0$ and $\Gamma =
\overline{\p\Omega \setminus \Gamma_0}$. Suppose that $\sigma_j, \mu_j
\in C^2(\overline\Omega)$, $j=1,2$, satisfy $\sigma_j \ge \sigma_0$
and $\mu_j \ge \mu_0$, for some constants $\sigma_0, \mu_0 > 0$, and
\begin{equation}\label{boundary assumption on Gamma}
   \p^\alpha \sigma_1|_{\Gamma} = \p^\alpha \sigma_2|_{\Gamma}
   \quad\text{and}\quad
   \p^\alpha \mu_1|_{\Gamma} = \p^\alpha \mu_2|_{\Gamma}
   \quad\text{for}\quad|\alpha| \le 2 .
\end{equation}
In addition, assume that $\sigma_j$ and $\mu_j$, $j=1,2$, can be
extended to $\R^3$ as $C^2$ functions which are invariant under
reflection across $\p B_0$. Suppose that $\omega > 0$ is not a
resonant frequency for $(\sigma_1,\mu_1)$ and $(\sigma_2,\mu_2)$. If
$$
   Z_{\sigma_1,\mu_1}^\omega(f)|_{\Gamma}
             = Z_{\sigma_2,\mu_2}^\omega(f)|_{\Gamma}
   \quad\text{for all}\quad f \in TH^{1/2}_{\Div}(\p\Omega)
   \quad\text{with}\quad \supp(f)\subset\Gamma ,
$$
then $\sigma_1 = \sigma_2$ and $\mu_1 = \mu_2$.
\end{Theorem}

\noindent
For the proof of Theorem~\ref{main thm}, we follow Isakov's reflection
approach \cite{isakov2007uniqueness} which was originally proposed for
the inverse conductivity problem. An analogous result for IEMP was
obtained in \cite{caro2009inverse}.

\smallskip

We briefly describe a connection of our results to the land-based CSEM
(controlled source electromagnetic) method in geophysical exploration
\cite{streich2016controlled}. Contrary to the MT method, the CSEM
method employs active sources. In recent work by Schaller \textit{et
  al.}~\cite{schaller2018land}, a land-based CSEM survey was designed
and performed at the Schoonebeek oil field. The application of
land-based CSEM for low-cost ${\rm CO}_2$ monitoring was studied in
\cite{mcaliley2019analysis}. The marine CSEM method, which was
introduced by Cox \textit{et al.}~\cite{cox1971electromagnetic}, would
require a careful incorporation of the ocean layer in the analysis,
which we do not pursue in this paper. It has successfull applications
in direct identification of hydrocarbons \cite{chave1991electrical,
  eidesmo2002sea}, and the study of the oceanic lithosphere and active
spreading centers \cite{chave1990some, constable1996marine,
  cox1986controlled, evans1994electrical, macgregor2002use,
  young1981electromagnetic}. For a more detailed exposition of
progress made on the marine CSEM, we refer to a review paper by
Constable \cite{constable2010ten}. Various basic data fitting
approaches have been developed for CSEM \cite{abubakar20082,
  gribenko2011joint, li20072dp1, li20072dp2}. From a mathematical
point of view, the data for the land-based CSEM is modeled by point
source measurements. More precisely, for an arbitrary unit vector
$\alpha$ and $y \in \p\Omega$, consider the equation
$$
   \nabla_x \times E_\alpha(x,y) = i \omega \mu(x) H_\alpha(x,y)
   + \delta(x-y) \nu(y) \times \alpha
   \quad\text{and}\quad
   \nabla_x \times H_\alpha(x,y) = \sigma(x) E_\alpha(x,y)
   \quad\text{in}\quad\R^3 ,
$$
with the outgoing radiation condition. The equation governs the
electromagnetic field of a magnetic dipole (active source) tangential
to the boundary $\p\Omega$. Then the inverse problem for the
land-based CSEM is to determine $\sigma$ and $\mu$ from
$$
   \mathcal A_{\sigma,\mu} := \left\{\big(\nu(x) \times H_{e_j}(x,y) ,\
   \nu(x) \times E_{e_j}(x,y)\big) : x, y \in\p\Omega ,\quad
   x \neq y ,\quad j=1,2,3\right\} ,
$$
where $e_j$, $j=1,2,3$, denote the Cartesian coordinate vectors. We
expect that following the arguments similar to
\cite{ola1996electromagnetic}, one can show that to the knowledge of
$\mathcal A_{\sigma,\mu}$ is equivalent the knowledge of the graph of
$Z^\omega_{\sigma,\mu}$ via layer potentials. Then the land-based CSEM
and MT would concern the same inverse problem with boundary
data.

\smallskip

The paper is organized as follows. In Section~\ref{section::well
  posedness}, we prove the well-posedness of the direct problem using
standard arguments. In Section~\ref{section::CGOs}, we first rewrite
\eqref{eqn::Maxwell} as the curl-curl equation and then construct
complex geometric optics solutions for it. We use these solutions to
prove Theorem~\ref{main thm} in Section~\ref{section::proof of main
  thm}. Next, in Section~\ref{section::reflection approach} we perform
the reflection approach of Isakov \cite{isakov2007uniqueness} and
prove an analog of Theorem~\ref{main thm 2} but in the case when the
part of the boundary inaccessible for measurements is a subset of the
plane $\{x\in\R^3 : x_3=0\}$. Theorem~\ref{main thm 2} is then proved
in Section~\ref{section::proof of thm 2} by analyzing the behavior of
\eqref{eqn::Maxwell} under the Kelvin
transform. Appendix~\ref{section::pullbacks} contains properties of
pullbacks used in the main text. Finally, in
Appendix~\ref{section::Appendix B} we show that the impedance map
$Z^\omega_{\sigma,\mu}$ is a pseudodifferential operator of order $1$
if $\sigma, \mu \in C^\infty(\overline\Omega)$. Using this fact, we
gain insight in the notion of apparent resistivity used in geophysics
from a mathematical point of view.

\section{Well-posedness of the direct problem}
\label{section::well posedness}

Let $\Omega\subset\R^3$ be a bounded domain with $C^{1,1}$ boundary as
before, and let $\sigma,\mu\in C^1(\overline\Omega)$ be such that
$\sigma \ge \sigma_0$ and $\mu \ge \mu_0$ for some constants
$\sigma_0, \mu_0 > 0$. Consider the following system of equations for
electromagnetic fields $E$ and $H$:
\begin{equation}\label{eqn::Maxwell homogenous in appendix}
\nabla\times E=i\omega\mu H\quad\text{and}\quad \nabla\times H=\sigma E \quad\text{in}\quad\Omega,
\end{equation}
with the tangential boundary condition $\mathbf{t}(H)=f$, where
$\omega$ is a complex number. The main result of the present section is

\begin{Theorem}\label{thm::well posedness new version homogeneous}
Let $\Omega\subset\R^3$ be a bounded domain with $C^{1,1}$ boundary and let $\sigma,\mu\in C^1(\overline\Omega)$ be such that $\sigma \ge \sigma_0$ and $\mu \ge \mu_0$ for some constants $\sigma_0, \mu_0 > 0$. There is a discrete subset $\Sigma$ of $\C$ such that for all $\omega\notin \Sigma$ and for a given $f\in TH_{\Div}^{1/2}(\p \Omega)$ the system \eqref{eqn::Maxwell homogenous in appendix} with $\mathbf t(H)=f$ has a unique solution $(E,H)\in H_{\Div}^1(\Omega)\times H_{\Div}^1(\Omega)$ satisfying
$$
\|E\|_{H_{\Div}^1(\Omega)}+\|H\|_{H_{\Div}^1(\Omega)}\le C\|f\|_{TH_{\Div}^{1/2}(\p \Omega)}
$$
for some constant $C>0$ independent of $f$.
\end{Theorem}

For $\omega>0$ with $\omega\notin \Sigma$, we define the \emph{impedance map} $Z_{\sigma,\mu}^\omega$ as
$$
Z_{\sigma,\mu}^\omega(f) := \mathbf t(E),\quad f\in TH_{\Div}^{1/2}(\p \Omega),
$$
where $(E,H)\in H_{\Div}^1(\Omega)\times H_{\Div}^1(\Omega)$ is the unique solution of the system \eqref{eqn::Maxwell homogenous in appendix} with $\mathbf t(H)=f$, guaranteed by Theorem~\ref{thm::well posedness new version homogeneous}. Moreover, the estimate provided in Theorem~\ref{thm::well posedness new version homogeneous} implies that the impedance map is a well-defined and bounded operator $Z_{\sigma,\mu}^\omega: TH_{\Div}^{1/2}(\p \Omega)\to TH_{\Div}^{1/2}(\p \Omega)$.\smallskip

To prove Theorem~\ref{thm::well posedness new version homogeneous}, we
consider the following non-homogeneous problem. Let $J_e$ and $J_m$ be
vector fields defined in $\Omega$ representing current sources. We
consider the non-homogenous time-harmonic Maxwell equations,
\begin{equation}\label{eqn::Maxwell in appendix}
\nabla\times E=i\omega\mu H + J_m\quad\text{and}\quad \nabla\times H=\sigma E + J_e \quad\text{in}\quad\Omega.
\end{equation}
We have

\begin{Theorem}\label{thm::well posedness new version}
Let $\Omega\subset\R^3$ be a bounded domain with $C^{1,1}$ boundary and let $\sigma,\mu\in C^1(\overline\Omega)$ be such that $\sigma \ge \sigma_0$ and $\mu \ge \mu_0$ for some constants $\sigma_0, \mu_0 > 0$. Suppose that $J_e,J_m\in L^2(\Omega;\C^3)$ such that $\nabla\cdot J_e,\nabla\cdot J_m\in L^2(\Omega;\C^3)$ and $\nu\cdot J_e|_{\p\Omega}, \nu\cdot J_m|_{\p\Omega}\in H^{1/2}(\p\Omega)$. Then there is a discrete subset $\Sigma$ of $\C$ such that for all $\omega\notin \Sigma$ the boundary value problem \eqref{eqn::Maxwell in appendix} with $\mathbf t(H) = 0$ has a unique solution $(E,H)\in H^1_{\Div}(\Omega)\times H^1_{\Div}(\Omega)$ satisfying
$$
\|E\|_{H_{\Div}^1(\Omega)}+\|H\|_{H_{\Div}^1(\Omega)}\le C\big(\|J_e\|_{L^2(\Omega;\C^3)}+\|J_m\|_{L^2(\Omega;\C^3)} + \|\nabla\cdot J_e\|_{L^2(\Omega)}+\|\nabla\cdot J_m\|_{L^2(\Omega)} + \|\nu\cdot J_e\|_{H^{1/2}(\p\Omega)}+\|\nu\cdot J_m\|_{H^{1/2}(\p\Omega)}\big)
$$
for some constant $C>0$ independent of $J_e$ and $J_m$.
\end{Theorem}

\noindent
We first prove Theorem~\ref{thm::well posedness new version} and then
show that it can be used to prove Theorem~\ref{thm::well posedness new
  version homogeneous}.

\subsection{Proof of Theorem~\ref{thm::well posedness new version}}

We introduce some notion that will be used for the proof. We work with the following Hilbert space which is the largest domain of $\nabla\times$:
$$
H(\curl;\Omega):=\{w\in L^2(\Omega;\C^3):\nabla\times w\in L^2(\Omega;\C^3)\}
$$
endowed with the norm $\|w\|_{H(\curl;\Omega)}:=\|w\|_{L^2(\Omega;\C^3)} + \|\nabla\times w\|_{L^2(\Omega;\C^3)}$. Then the tangential trace operator has its extensions to bounded operators $\mathbf{t}:H(\curl; \Omega)\to H^{-1/2}(\p \Omega;\C^3)$. We also work with the space of vector fields in $H(\curl;\Omega)$ having zero tangential trace
$$
H(\curl, 0; \Omega):=\{w\in H(\curl;\Omega):\mathbf{t}(w)=0\}.
$$
For the short proof, we follow the standard variational methods used in \cite{kirsch2016mathematical, monk2003finite}. Substituting the first equation of \eqref{eqn::Maxwell in appendix} into the second one, we obtain
\begin{equation}\label{eqn::second order equation}
\nabla\times(\sigma^{-1}\nabla\times H) - i\omega\mu H = J_m +\nabla\times(\sigma^{-1} J_e)\quad\text{in}\quad\Omega.
\end{equation}
Our first step is to find a unique solution $H\in H(\curl, 0;\Omega)$ of this equation satisfying
\begin{equation}\label{eqn::estimate for H}
\|H\|_{H(\curl;\Omega)}\le C(\|J_e\|_{L^2(\Omega;\C^3)}+\|J_m\|_{L^2(\Omega;\C^3)}).
\end{equation}
By Helmholtz type decompositions in \cite[Section~4.1.3]{kirsch2016mathematical} or \cite[Section~3.7]{monk2003finite}, we can uniquely decompose
\begin{align*}
H &= H_0+\nabla h,\quad H_0\in H(\curl, 0;\Omega)_{0,\mu} := \{w\in H(\curl, 0; \Omega): \nabla\cdot(\mu w) = 0\},\quad h\in H^1_0(\Omega;\C),\\
\mu^{-1}J_m &= J_{m,0}+\nabla j_m,\quad J_{m,0}\in L^2(\Omega;\C^3)_{0,\mu} := \{w\in L^2(\Omega;\C^3): \nabla\cdot(\mu w) = 0\},\quad j_m\in H^1_0(\Omega;\C).
\end{align*}
We note here that
\begin{equation}\label{eqn::estimate for j_e in H^1 norm}
\|j_m\|_{H^1(\Omega;\C)}\le C\|J_m\|_{L^2(\Omega;\C^3)}.
\end{equation}
Using these decompositions, \eqref{eqn::second order equation} can be rewritten as
\begin{equation}\label{eqn::second order equation in a weak form -- rewritten}
\nabla\times(\sigma^{-1}\nabla\times H_0) - i\omega\mu H_0 - i\omega\mu \nabla h = \mu J_{m,0} + \mu\nabla j_m +\nabla\times(\sigma^{-1} J_e)\quad\text{in}\quad\Omega.
\end{equation}
To extract $h$ from \eqref{eqn::second order equation in a weak form -- rewritten}, we simply set $h=-(i\omega)^{-1}j_m$. Thus, we need to find a unique $H_0\in H(\curl, 0;\Omega)$ with $\nabla\cdot(\mu H_0) = 0$ satisfying
\begin{equation}\label{eqn::second order equation -- rewritten}
\nabla\times(\sigma^{-1}\nabla\times H_0) - i\omega\mu H_0 = \mu J_{m,0} + \nabla\times(\sigma^{-1} J_e)\quad\text{in}\quad\Omega.
\end{equation}
To solve this equation, we need the following result on existence of a solution operator
\begin{Proposition}\label{prop::resonances}
There exist a constant $\lambda>0$ and a bounded linear map $T_\lambda : H(\curl, 0;\Omega)'\to H(\curl, 0;\Omega)$ such that
\begin{equation}\label{eqn::T_lambda is an inverse of a certain PDoperator}
\nabla \times (\sigma^{-1} \nabla \times T_\lambda u) + \lambda \mu T_\lambda u = u,\quad u\in H(\curl, 0;\Omega)'
\end{equation}
and
$$
T_\lambda(\nabla\times(\sigma^{-1}\nabla\times e)+\lambda\mu e)=e,\quad e\in H(\curl, 0;\Omega).
$$
Furthermore, if $\nabla\cdot u = 0$, then $T_\lambda u\in H(\curl, 0;\Omega)_{0,\mu}$. 
\end{Proposition}

Here and in what follows, $\langle\cdot,\cdot\rangle_\Omega$ is the duality between $H(\curl, 0;\Omega)'$ and $H(\curl, 0;\Omega)$ naturally extending the inner product of $L^2(\Omega;\C^3)$.

\begin{proof}
The proof is similar to that of \cite[Proposition~5.1]{assylbekov2016note} using Lax-Milgram's lemma.
\end{proof}

Then $H_0\in H(\curl, 0;\Omega)$ with $\nabla\cdot(\mu H_0) = 0$ solves \eqref{eqn::second order equation -- rewritten} if and only if
\begin{equation}\label{eqn::second order equation -- rewritten in terms of T_lambda operators}
H_0-(i\omega+\lambda)\widetilde T_{\lambda}H_0=T_\lambda\left(\mu J_{m,0}+\nabla\times(\sigma^{-1}J_e)\right)
\end{equation}
where $\widetilde T_{\lambda}=T_\lambda \circ m_\mu\circ P_\mu$, $m_\mu$ is multiplication by $\mu$, and $P_\mu$ is the bounded orthogonal projection of $L^2(\Omega;\C^3)$ onto $L^2(\Omega;\C^3)_{0,\mu}$ constructed in \cite[Section~4.1.3]{kirsch2016mathematical}. Since $J_{m,0}\in L^2(\Omega;\C^3)_{0,\mu}$ and $\nabla\cdot\nabla\times = 0$, we then have $\nabla\cdot\left(\mu J_{m,0}+\nabla\times(\sigma^{-1}J_e)\right) = 0$. Therefore, by the second part of Proposition~\ref{prop::resonances}, this implies that $T_\lambda\left(\mu J_{m,0}+\nabla\times(\sigma^{-1}J_e)\right)$ belongs to $H(\curl, 0;\Omega)_{0,\mu}$. The second part of Proposition~\ref{prop::resonances} implies also that $\widetilde T_{\lambda}$ can be considered as a bounded linear operator
$$
\widetilde T_{\lambda} : L^2(\Omega;\C^3)_{0,\mu}\overset{m_\mu}{\longrightarrow} L^2(\Omega;\C^3)_{0,1}\overset{T_\lambda}{\longrightarrow} H(\curl, 0;\Omega)_{0,\mu}\hookrightarrow L^2(\Omega;\C^3)\overset{P_\mu}{\longrightarrow} L^2(\Omega;\C^3)_{0,\mu}.
$$
Using the compactness of the inclusion $H(\curl, 0;\Omega)_{0,\mu}\hookrightarrow L^2(\Omega;\C^3)$ \cite{weber1980local} and following similar reasoning as at the end of \cite[Section~5]{assylbekov2016note}, one can show that  for any $\omega\notin \Sigma$, where $\Sigma:=\{\omega\in \C\setminus\{\pm i\lambda\}:(i\omega+\lambda)^{-1}\in \Spec(\widetilde T_{\lambda})\}$ which is discrete, \eqref{eqn::second order equation -- rewritten in terms of T_lambda operators} has a unique solution $H_0\in H(\curl, 0;\Omega)_{0,\mu}$ satisfying
$$
\|H_0\|_{H(\curl;\Omega)}\le C(\|J_e\|_{L^2(\Omega;\C^3)}+\|J_{m,0}\|_{L^2(\Omega;\C^3)}).
$$
Next, setting $H = H_0-(i\omega)^{-1}j_m$, we obtain a unique $H(\curl, 0;\Omega)$ solution for \eqref{eqn::second order equation} satisfying \eqref{eqn::estimate for H} thanks to \eqref{eqn::estimate for j_e in H^1 norm}. Defining $E := \sigma^{-1}(\nabla\times H - J_e)$ we obtain a unique $(E,H)\in H(\curl, 0;\Omega)\times H(\curl;\Omega)$ solving \eqref{eqn::Maxwell in appendix} and satisfying
$$
\|E\|_{H(\curl;\Omega)} + \|H\|_{H(\curl;\Omega)}\le C(\|J_e\|_{L^2(\Omega;\C^3)}+\|J_m\|_{L^2(\Omega;\C^3)}).
$$
To prove that $(E,H)\in H^1_{\Div}(\Omega)\times H^1_{\Div}(\Omega)$, apply $\nabla\cdot$ to \eqref{eqn::Maxwell in appendix} and get $\nabla\cdot(i\omega\mu H) =-\nabla\cdot J_m$ and $\nabla\cdot(\sigma E) = -\nabla\cdot J_e$. Hence $\nabla\cdot E, \nabla\cdot H\in L^2(\Omega)$, since $\nabla\cdot J_e, \nabla\cdot J_m\in L^2(\Omega)$ by assumption. Then $\bt(H)=0$ and the results in \cite{costabel1990remark} imply that $H\in H^1_{\Div}(\Omega)$ and
$$
\|H\|_{H^1_{\Div}(\Omega)} \le C\big(\|H\|_{H(\curl;\Omega)} + \|J_e\|_{L^2(\Omega;\C^3)} + \|J_m\|_{L^2(\Omega;\C^3)} + \|\nabla\cdot J_e\|_{L^2(\Omega)} + \|\nabla\cdot J_m\|_{L^2(\Omega)}\big).
$$
To show that $E\in H^1(\Omega;\C^3)$, observe that $\nu\cdot (\nabla\times H)|_{\p\Omega} = - \Div(\bt(H))=0$ by \cite[Corollary~A.20]{kirsch2016mathematical}. Then by \eqref{eqn::Maxwell in appendix}, $\nu\cdot E|_{\p\Omega} = \sigma^{-1} \nu\cdot (\nabla\times H)|_{\p\Omega} - \sigma^{-1} \nu\cdot J_e|_{\p\Omega} = - \sigma^{-1} \nu\cdot J_e|_{\p\Omega} \in H^{1/2}(\p\Omega)$. According to the results in \cite{costabel1990remark}, this implies that $E\in H^1(\Omega;\C^3)$ and
$$
\|E\|_{H^1_{\Div}(\Omega)} \le C\big(\|E\|_{H(\curl;\Omega)} + \|J_e\|_{L^2(\Omega;\C^3)} + \|\nabla\cdot J_e\|_{L^2(\Omega)} + \|\nu\cdot J_e\|_{H^{1/2}(\p\Omega)}\big).
$$
Next, using \cite[Corollary~A.20]{kirsch2016mathematical} and \eqref{eqn::Maxwell in appendix}, we can show $\Div(\bt(E)) = - \nu\cdot (\nabla\times E)|_{\p\Omega} = - i\omega\mu \nu\cdot H|_{\p\Omega} + \nu\cdot J_m|_{\p\Omega} \in H^{1/2}(\p\Omega)$. Thus, $E\in H^1_{\Div}(\Omega)$. Finally, the estimate in the statement of the theorem follows by combining all the above estimates. The proof of Theorem~\ref{thm::well posedness new version} is thus complete.

\subsection{Proof of Theorem~\ref{thm::well posedness new version homogeneous}}

First prove the uniqueness of the solution. For a fixed $\omega\in
\C$, suppose that $(E_j,H_j)\in H^1_{\Div}(\Omega)\times
H^1_{\Div}(\Omega)$, $j=1,2$, solve \eqref{eqn::Maxwell homogenous in
  appendix} and satisfy $\mathbf{t}(H_1)=\mathbf{t}(H_2)$. Then
$(E,H)\in H^1_{\Div}(\Omega)\times H^1_{\Div}(\Omega)$ also solve
\eqref{eqn::Maxwell homogenous in appendix} and satisfy
$\mathbf{t}(H)=0$, where $E:=E_1-E_2$ and $H:=H_1-H_2$. The uniqueness
part of Theorem~\ref{thm::well posedness new version} (with
$J_e=J_m=0$) gives that $E=0$ and $H=0$.\smallskip

Next, we prove existence of a solution. For a given $f\in TH^{1/2}_{\Div}(\p\Omega)$, there is $H'\in H^1_{\Div}(\Omega)$ such that $\mathbf{t}(H')=f$ and $\|H'\|_{H^1_{\Div}(\Omega)}\le C\|f\|_{TH^{1/2}_{\Div}(\p\Omega)}$. Applying Theorem~\ref{thm::well posedness new version} with $J_e=-\nabla\times H'$ and $J_m=i\omega\mu H'$, we obtain a unique $(E_0,H_0)\in H^1_{\Div}(\Omega)\times H^1_{\Div}(\Omega)$ solving
$$
\nabla\times E_0=i\omega\mu H_0+i\omega\mu H',\quad \nabla\times H_0=\sigma E_0 - \nabla\times H',\quad \bt(H_0)=0
$$
and satisfying $\|E_0\|_{H_{\Div}^1(\Omega)}+\|H_0\|_{H_{\Div}^1(\Omega)}\le C \|f\|_{TH^{1/2}_{\Div}(\p\Omega)}$. Here, we used the fact that $\nu\cdot (\nabla\times H')|_{\p\Omega} = - \Div(\bt(H'))\in H^{1/2}(\p\Omega)$ by \cite[Corollary~A.20]{kirsch2016mathematical}. Then $(E,H)\in H^1_{\Div}(\Omega)\times H^1_{\Div}(\Omega)$ solves \eqref{eqn::Maxwell homogenous in appendix} with $\mathbf{t}(E)=f$, where $E:=E_0+E'$ and $H:=H_0$. The proof is complete.

\section{Construction of complex geometric optics solutions}\label{section::CGOs}

Throughout this section, we assume that $\sigma$ and $\mu$ can be extended to the whole $\R^3$ so that $\sigma\ge \sigma_0$, $\mu\ge \mu_0$ and
\begin{equation}\label{sigma and mu are constants outside of a compact set}
\sigma - \sigma_0,\quad \mu - \mu_0 \in C^2_0(\R^3)
\end{equation}
for some constants $\sigma_0, \mu_0 > 0$. We also let $R>0$ be large enough (but fixed) so that $B_R(0)$ contains both $\supp(\sigma - \sigma_0)$ and $\supp(\mu - \mu_0)$.\smallskip

Substituting the first equation of \eqref{eqn::Maxwell} into the second one, we obtain the following second-order equation
\begin{equation}\label{eqn::curl-curl equation}
\nabla\times(\sigma^{-1}\nabla\times H) - i\omega \mu H = 0\quad\text{in}\quad \Omega.
\end{equation}
The aim of the present section is to construct a complex geometric optics solution in $H\in H_{\Div}^1(\Omega)$ for the above equation. Instead of working in $\Omega$, we conduct our analysis in the whole $\R^3$. Therefore, we consider
\begin{equation}\label{eqn::curl-curl equation R^3}
\nabla\times(\sigma^{-1}\nabla\times H) - i\omega \mu H = 0\quad\text{in}\quad \R^3.
\end{equation}

Taking the divergence of \eqref{eqn::curl-curl equation R^3}, it straightforwardedly follows that $\nabla\cdot(\mu H) = 0$ in $\R^3$. Therefore, we obtain
$$
\nabla\times\nabla\times H =  - \Delta H - \nabla(\nabla\beta \cdot H)\quad\text{in}\quad \R^3,
$$
where $\beta := \log\mu$. Then, we use the latter identity in \eqref{eqn::curl-curl equation R^3} to show that this equation is equivalent to the system
\begin{align}
L_{\sigma, \mu} H := -\Delta H - \nabla(\nabla\beta \cdot H) - \nabla\alpha \times \nabla\times H - i\omega \sigma \mu H = 0\quad\text{in}\quad \R^3, \label{eqn3-1}\\
\nabla\cdot(\mu H) = 0\quad\text{in}\quad \R^3, \label{eqn3-2}
\end{align}
where $\alpha := \log\sigma$. We note that the derivatives $\p^\kappa \alpha$ and $\p^\kappa \beta$ are uniformly continuous on $\R^3$ for $|\kappa|=0, 1, 2$.\smallskip

The complex geometric optics solutions we aim to construct are of the form
$$
H(x; \zeta) = e^{i \zeta\cdot x}(a(x; \zeta) + r(x; \zeta)),
$$
where $\zeta\in \C^3\setminus\{0\}$ such that $\zeta\cdot\zeta = i\omega \sigma_0 \mu_0$, $a$ is a specific complex-valued smooth vector field on $\R^3$ and $r$ is the correction term. Then \eqref{eqn3-1} is equivalent to
\begin{equation} \label{eq:Lsm}
e^{-i \zeta\cdot x}L_{\sigma,\mu}(e^{i \zeta\cdot x}r) = - f,\quad f := e^{-i \zeta\cdot x} L_{\sigma,\mu}(e^{i \zeta\cdot x}a).
\end{equation}

\subsection{Solution operator}

For $\zeta\in \C^3\setminus\{0\}$ such that $\zeta\cdot\zeta = i\omega \sigma_0 \mu_0$, we define the operators
$$
\nabla_\zeta := \nabla + i\zeta,\quad \Delta_\zeta := \Delta + 2i \zeta\cdot \nabla.
$$
Then
\begin{equation}\label{eqn:: conj grad and laplacian}
e^{-i \zeta\cdot x}\circ \nabla\circ e^{i \zeta\cdot x} = \nabla_\zeta\quad\text{and}\quad e^{-i \zeta\cdot x}\circ \Delta\circ e^{i \zeta\cdot x} = \Delta_\zeta - i\omega \sigma_0 \mu_0.
\end{equation}
For $\delta\in\R$, we define the $L^2$-based weighted space on $\R^3$
$$
L^2_\delta := \left\{f: \R^3 \to \C^3: \|f\|_{L^2_\delta}:=\Big(\int_{\R^3} (1+|x|^2)^{\delta}|f(x)|^2\,dx\Big)^{1/2} < \infty\right\}
$$
and
$$
H_\delta^1 := \left\{f \in L^2_\delta: \|f\|_{H^1_\delta}:=\|f\|_{L^2_\delta} + \sum_{j=1}^3\|\p_j f\|_{L^2_\delta} < \infty\right\}.
$$

\begin{Proposition}\label{prop unique solvability}
For $k\in \C$, suppose $\zeta\in \C^3$ with $\zeta \cdot \zeta = k$, $-1 < \delta < 0$. Assume that $\gamma\in C^2(\R^3)$ is positive. Then for $f\in L^2_{\delta+1}$ there is a unique  $u\in L^2_\delta$ solving
\begin{equation}\label{eqn elliptic}
(-\Delta_\zeta - \nabla\log\gamma \cdot \nabla_\zeta) u = f\quad\text{in}\quad\R^3
\end{equation}
such that
$$
\|u\|_{L^2_\delta} \le \frac{C}{|\zeta|}\|f\|_{L^2_{\delta+1}}
$$
for some constant $C>0$. Furthermore, $u$ belongs to $H^1_\delta$.
\end{Proposition}

\begin{proof}
It follows from the identity
$$
(-\Delta_\zeta - \nabla\log\gamma \cdot \nabla_\zeta) u = \gamma^{-1/2}(-\Delta_\zeta + q) (\gamma^{1/2}u),\quad q:=\gamma^{-1/2}\Delta\gamma^{1/2}\in C_0(\R^3),
$$
that solving \eqref{eqn elliptic} is equivalent to solving
$$
(-\Delta_\zeta + q) \tilde u = \gamma^{1/2} f\quad\text{in}\quad\R^3,
$$
where $\tilde u := \gamma^{1/2} u$. By \cite[Theorem~1.6]{sylvester1987global} there is a unique $\tilde u\in L^2_\delta$  solving the above equation and satisfying
$$
\|\tilde u\|_{L^2_\delta} \le \frac{C}{|\zeta|}\|\gamma^{1/2}f\|_{L^2_{\delta+1}}.
$$
Next, \cite[Lemma~1.15]{sun1992inverse} implies that $\tilde u \in H^1_\delta$. 
The result now follows immediately by setting $u = \gamma^{-1/2} \tilde u$.
\end{proof}

According to Proposition~\ref{prop unique solvability}, for sufficiently large $|\zeta|$, there is a bounded inverse $G_{\zeta, \gamma}: L^2_{\delta + 1} \to L^2_\delta$ of $-\Delta_\zeta - \nabla\log\gamma \cdot \nabla_\zeta$ such that
$$
\|G_{\zeta,\gamma}\|_{L^2_{\delta + 1} ; L^2_\delta} = \mathcal O\Big(\frac{1}{|\zeta|}\Big)\quad\text{as}\quad |\zeta|\to\infty.
$$
Moreover, $G_{\zeta, \gamma}$ maps $L^2_{\delta + 1}$ into $H^1_\delta$.

\subsection{Mollified $\sigma$ and $\mu$}
Let $\Phi\in C^\infty_0(\R^3)$ with $0 \le \Phi \le 1$ and $\int_{\R^3} \Phi(x)\,dx=1$. For a fixed $\epsilon$ with $0 < \epsilon < 1/8$, we consider
$$
\Phi_\tau(x) := \Big(\frac{1}{\tau^{-\epsilon}}\Big)^3 \Phi\Big(\frac{x}{\tau^{-\epsilon}}\Big)\quad\text{for large } \tau>0.
$$
We define 
$$
\alpha^\sharp(x;\tau) := \alpha * \Phi_\tau (x),\quad \beta^\sharp(x;\tau) := \beta * \Phi_\tau(x)\quad\text{for}\quad x\in \R^3.
$$
Then $\alpha^\sharp(\cdot;\tau), \beta^\sharp(\cdot;\tau) \in C^\infty(\R^3)$. From
$\p^\kappa \alpha^\sharp = (\p^\kappa \alpha) * \Phi_\tau$ and $\p^\kappa \beta^\sharp = (\p^\kappa \beta) * \Phi_\tau$,
it follows that
\begin{equation}\label{supports of derivatives of alpha and beta}
\supp(\p^\kappa \alpha^\sharp) ,\ \supp(\p^\kappa \beta^\sharp) \subset B_{R+2\tau^{-\epsilon}}(0),\quad |\kappa| = 1, 2.
\end{equation}
We also have
\begin{equation}\label{L^infty norms of derivatives of alpha and beta errors}
\|\alpha -  \alpha^\sharp\|_{W^{2,\infty}(\R^3)} = o(1), \quad \| \beta - \beta^\sharp\|_{W^{2,\infty}(\R^3)} = o(1)\quad\text{as}\quad \tau\to\infty.
\end{equation}
Indeed,
$$
\p^\kappa \alpha(x) - \p^\kappa \alpha^\sharp(x;\tau) = \int_{\R^3} \Phi(y) [\p^\kappa\alpha(x) - \p^\kappa\alpha(x - \tau^{-\epsilon} y)]\,dy,\quad |\kappa|=0, 1, 2,
$$
and uniform continuity of $\p^\kappa\alpha$ on $\R^3$ gives the desired estimate using \cite[Theorem~0.13]{folland1995introduction}. A similar argument can be used for $\beta$.\smallskip

Finally, using that $\p^\kappa \alpha^\sharp = \alpha * (\p^\kappa\Phi_\tau)$ and $\p^\kappa \beta^\sharp = \beta * (\p^\kappa\Phi_\tau)$, a direct calculation shows that
$$
\|\p^\kappa\alpha^\sharp\|_{L^{\infty}(\R^3)},\ \|\p^\kappa\beta^\sharp\|_{L^{\infty}(\R^3)} = \mathcal O_\kappa(\tau^{|\kappa|\epsilon})\quad\text{for}\quad |\kappa|\ge 0\quad\text{as}\quad \tau\to \infty,
$$
which implies that
\begin{equation}\label{eqn::L^infty norms of derivatives of alpha & beta}
\|\alpha^\sharp\|_{W^{k,\infty}(\R^3)},\ \|\beta^\sharp\|_{W^{k,\infty}(\R^3)} = \mathcal O_k(\tau^{k\epsilon})\quad\text{for}\quad k = 0,1,2,\ldots\quad\text{as}\quad \tau\to \infty.
\end{equation}
We note that stronger estimates follow from \eqref{L^infty norms of derivatives of alpha and beta errors}
\begin{equation}\label{eqn::W^2,infty norms of alpha & beta sharps}
\|\alpha^\sharp\|_{W^{2,\infty}(\R^3)},\ \|\beta^\sharp\|_{W^{2,\infty}(\R^3)} = \mathcal O(1)\quad\text{as}\quad \tau\to \infty.
\end{equation}

\subsection{Transport equation}

Since $\zeta \cdot \zeta = i\omega \sigma_0 \mu_0$, using \eqref{eqn:: conj grad and laplacian}, we obtain (cf.~(\ref{eq:Lsm}))
\begin{align*}
f &= -\Delta_\zeta a - \nabla_\zeta(\nabla\beta\cdot a) - \nabla\alpha\times \nabla_\zeta\times a  - i\omega\sigma\mu a + i\omega \sigma_0 \mu_0 a \\
& = -\Delta a - \nabla(\nabla\beta\cdot a) - \nabla\alpha\times \nabla\times a  - i\omega (\sigma\mu - \sigma_0 \mu_0) a
- 2i \zeta\cdot \nabla a - (\nabla\beta\cdot a) i\zeta - \nabla\alpha \times (i\zeta \times a).
\end{align*}
We shall consider $\zeta$ of the form $\zeta = \tau \rho + \zeta_1$ where $\tau > 0$ is a large parameter, $\rho \in \C^3$  is independent of $\tau$ and satisfies $\Re\rho \cdot \Im\rho = 0$ and $|\Re\rho| = |\Im\rho| = 1$, and $\zeta_1 = \mathcal O(1)$ as $\tau \to \infty$. Then
\begin{align*}
f &= -\Delta a - \nabla(\nabla\beta\cdot a) - \nabla\alpha\times \nabla\times a  - i\omega (\sigma\mu - \sigma_0 \mu_0) a
- 2i \zeta_1\cdot \nabla a - (\nabla\beta\cdot a) i\zeta_1 - \nabla\alpha \times (i\zeta_1 \times a)\\
& \quad - i\tau ((\nabla(\beta-\beta^\sharp)\cdot a) \rho + \nabla(\alpha-\alpha^\sharp) \times (\rho \times a))
 - i\tau (2\rho\cdot \nabla a + (\nabla\beta^\sharp\cdot a) \rho + \nabla\alpha^\sharp \times (\rho \times a)).
\end{align*}
In order to get $\|f \|_{L^2_{\delta+1}} = o(\tau)$ as $\tau \to
\infty$, for $-1 < \delta < 0$, we should construct $a$ satisfying the
transport equation, that is,
\begin{equation}\label{transport eqn}
2\rho\cdot \nabla a + (\nabla\beta^\sharp\cdot a) \rho + \nabla\alpha^\sharp \times (\rho \times a) = o(1)\quad\text{in}\quad L^2_{\delta+1}\quad\text{as}\quad\tau\to\infty.
\end{equation}
We
get
\begin{equation}\label{transport eqn rewritten}
2\rho \cdot \nabla a + (\nabla\beta^\sharp \cdot a) \rho + (\nabla\alpha^\sharp \cdot a) \rho - (\rho \cdot \nabla\alpha^\sharp) a = o(1)\quad\text{in}\quad L^2_{\delta+1}\quad\text{as}\quad\tau\to\infty.
\end{equation}
For arbitrary $s_0\in \R$, we seek solutions of \eqref{transport eqn rewritten} in the form
$$
a = e^{-\alpha^\sharp / 2} \rho + s_0 e^{\alpha^\sharp / 2}e^{\chi_\tau \Psi^\sharp} \overline{\rho},
$$
where $\chi_\tau(x) := \chi(\tau^{-\theta}x)$, $0< \theta < 1/2$ with $\chi\in C^\infty_0(\R^3)$ such that $\chi(x) \equiv 1$ for $|x| < 1/2$ and $\chi(x) \equiv 0$ for $|x| \ge 1$. Then
\begin{equation}\label{nabla a}
\p_j a = - \frac{1}{2} \p_j \alpha^\sharp e^{-\alpha^\sharp / 2} \rho + \frac{s_0}{2}\p_j \alpha^\sharp e^{\alpha^\sharp / 2}e^{\chi_\tau \Psi^\sharp} \overline{\rho}
 + \frac{s_0}{\tau^\theta} (\p_j\chi)(\theta^{-\theta} x)\Psi^\sharp e^{\alpha^\sharp / 2}e^{\chi_\tau \Psi^\sharp} \overline{\rho} + s_0 \chi_\tau \p_j \Psi^\sharp e^{\alpha^\sharp / 2}e^{\chi_\tau \Psi^\sharp} \overline{\rho}.
\end{equation}
Substituting these into \eqref{transport eqn rewritten}, we come to
$$
s_0 \Big\{ \tau^{-\theta}\Psi^\sharp 2\rho\cdot(\nabla \chi)(\tau^{-\theta}\,\cdot\,) + \chi_\tau 2\rho \cdot \nabla \Psi^\sharp + \overline{\rho} \cdot \nabla(\alpha^\sharp + \beta^\sharp) \Big\} e^{\alpha^\sharp / 2} e^{\chi_\tau \Psi^\sharp} \overline{\rho} = o(1)\quad\text{in}\quad L^2_{\delta+1}\quad
\text{as $\tau\to\infty$}. 
$$
We observe that by \eqref{supports of derivatives of alpha and beta}, $\nabla(\alpha^\sharp + \beta^\sharp) = \chi_\tau \nabla(\alpha^\sharp + \beta^\sharp)$ for large enough $\tau>0$. Therefore, it is enough to find $\Psi^\sharp$ in $C^\infty(\R^3)$ with a nice control on $\|\Psi^\sharp\|_{L^\infty(\R^3)}$ and satisfying
\begin{equation}\label{eqns for Phi and Psi} 
2\rho \cdot \nabla \Psi^\sharp + \overline{\rho} \cdot \nabla(\alpha^\sharp + \beta^\sharp) = 0 \quad\text{in}\quad \R^3.
\end{equation}
Since $\rho\cdot \rho = 0$ and $\Re\rho \cdot \Im\rho = 0$, the operator $N_\rho := \rho \cdot \nabla$ is just the $\overline\p$-operator in certain linear coordinates. Its inverse is defined as
$$
N_\rho^{-1} f(x) := \frac{1}{2\pi} \int_{\R^2} \frac{f(x - y_1 \Re\rho - y_2 \Im\rho)}{y_1 + i y_2}\,dy,\quad f\in C_0(\R^3).
$$
Then by \cite[Lemma~4.6]{salo2004inverse},
$$
\Psi^\sharp(x,\rho;\tau) := - \frac{1}{2} N_\rho^{-1} \{\overline{\rho} \cdot \nabla(\alpha^\sharp + \beta^\sharp)\} \in C^{\infty}(\R^3)
$$
satisfies equation \eqref{eqns for Phi and Psi}. It also follows from \cite[Lemma~4.6]{salo2004inverse} and \eqref{L^infty norms of derivatives of alpha and beta errors} that
\begin{equation}\label{estimate W^1,infty of Psi^sharp}
\|\Psi^\sharp\|_{W^{1,\infty}(\R^3)} \le \mathcal O(1) \big\{\|\alpha^\sharp\|_{W^{2,\infty}(\R^3)} + \|\beta^\sharp\|_{W^{2,\infty}(\R^3)} \big\} = \mathcal O(1)\quad\text{as}\quad \tau\to\infty.
\end{equation}
Furthermore, \cite[Lemma~4.6]{salo2004inverse} and \eqref{eqn::L^infty norms of derivatives of alpha & beta} imply that
\begin{equation}\label{estimate W^k,infty of Psi^sharp}
\begin{aligned}
\|\Psi^\sharp\|_{W^{k,\infty}(\R^3)} &\le \mathcal O_k(1) \big\{\|\alpha^\sharp\|_{W^{k+1,\infty}(\R^3)} + \|\beta^\sharp\|_{W^{k+1,\infty}(\R^3)} \big\}\\
&= \mathcal O_k(\tau^{k\epsilon+\epsilon})\quad\text{for}\quad k = 0,1,\ldots \quad\text{as}\quad \tau\to\infty.
\end{aligned}
\end{equation}
We set
$$
\Psi(x,\rho) := -\frac{1}{2} N_\rho^{-1} \{\overline{\rho} \cdot \nabla(\alpha + \beta)\} \in L^\infty(\R^3).
$$
Then \cite[Lemma~3.1]{sylvester1987global} together with \eqref{supports of derivatives of alpha and beta} and \eqref{L^infty norms of derivatives of alpha and beta errors} implies that
$$
\|\Psi^\sharp(\cdot,\rho;\tau) - \Psi(\cdot,\rho)\|_{L^2_{\rm loc}(\R^3)} = o(1)\quad\text{as}\quad \tau\to \infty.
$$
Finally, by \cite[Lemma~3.1]{salo2006semiclassical} together with \eqref{eqn::L^infty norms of derivatives of alpha & beta} and \eqref{eqn::W^2,infty norms of alpha & beta sharps},
\begin{equation}\label{estimate pointwise derivatives of Psi^sharp}
\begin{aligned}
|\p^\kappa \Psi^\sharp(x,\rho;\tau)| &\le \mathcal O_\kappa(1) \begin{cases}
(1 + |x_T|^2)^{-1/2}\chi_{B(0,R)}(x_\perp),&\text{if } |\kappa| = 0, 1,\\
\tau^{|\kappa|\epsilon + \epsilon}(1 + |x_T|^2)^{-1/2}\chi_{B(0,R)}(x_\perp),&\text{otherwise}
\end{cases}\\
&\quad\,\,\text{as}\quad \tau\to\infty,
\end{aligned}
\end{equation}
where $x_T$ is the projection of $x$ onto $\Span\{\rho_1,\rho_2\}$ and $x_\perp = x - x_T$, and $R>0$ is such that $B(0,R)$ contains $\supp(\sigma - \sigma_0)$ and $\supp(\mu - \mu_0)$.\smallskip

It follows that
\begin{equation}\label{integral Japanese symbol estimate}
\int_{|x_T| \le \tau^\theta,\, |x_\perp| \le R} (1 + |x_T|^2)^{-1} (1 + |x|^2)^{\delta+1}\, dx
\le \mathcal O(1) \int_{|x_T| \le \tau^\theta} (1 + |x_T|^2)^{\delta}\, dx_T = \mathcal O(\tau^{2(\delta+1)\theta})\quad\text{as}\quad\tau \to \infty.
\end{equation}
Similarly,
\begin{equation}\label{integral Japanese symbol estimate -j}
\int_{|x_T| \le \tau^\theta,\, |x_\perp| \le R} (1 + |x_T|^2)^{-j} (1 + |x|^2)^{\delta+1}\, dx
\le \mathcal O(1) \Big(1-\frac{1}{(1+\tau^{2\theta})^{j-2-\delta}}\Big) = \mathcal O(1)\quad\text{as}\quad\tau \to \infty\quad\text{for}\quad j\ge 2.
\end{equation}
We have

\begin{Lemma}
Let $\chi_\tau$ and $\Psi^\sharp$ be as above. Then
\begin{align}
\|\p_j(\chi_\tau \Psi^\sharp)\|_{L^2_{\delta+1}} &= \mathcal O(\tau^{(\delta+1)\theta}),\label{lemma tech 1}\\
\|\p_j(\chi_\tau \Psi^\sharp) \p_k(\chi_\tau \Psi^\sharp)\|_{L^2_{\delta+1}} &= \mathcal O(1),\label{lemma tech 2}\\
\|\p_j(\chi_\tau \Psi^\sharp) \p_k(\chi_\tau \Psi^\sharp) \p_l(\chi_\tau \Psi^\sharp)\|_{L^2_{\delta+1}} &= \mathcal O(1),\label{lemma tech 3}\\
\|\p_j\p_k(\chi_\tau \Psi^\sharp)\|_{L^2_{\delta+1}} &= o(\tau),\label{lemma tech 4}\\
\|\p_j\p_k(\chi_\tau \Psi^\sharp) \p_l(\chi_\tau \Psi^\sharp)\|_{L^2_{\delta+1}} &= \mathcal O(\tau^{3\epsilon}),\\
\|\p_l\p_j\p_k(\chi_\tau \Psi^\sharp)\|_{L^2_{\delta+1}} &= o(\tau^{1+\epsilon})\label{lemma tech 5}
\end{align}
as $\tau \to \infty$.
\end{Lemma}

\begin{proof}
Using \eqref{estimate pointwise derivatives of Psi^sharp} and \eqref{integral Japanese symbol estimate}, we obtain
\begin{align*}
\|\p_j(\chi_\tau \Psi^\sharp)\|_{L^2_{\delta+1}} &\le \frac{1}{\tau^\theta}\|\p_j\chi(\tau^{-\theta}\,\cdot\,) \Psi^\sharp\|_{L^2_{\delta+1}} + \|\chi_\tau \p_j\Psi^\sharp\|_{L^2_{\delta+1}} \\
&\le \mathcal O\Big(\frac{1}{\tau^\theta} + 1\Big) \int_{|x_T| \le \tau^\theta,\, |x_\perp| \le R} (1 + |x_T|^2)^{-1} (1 + |x|^2)^{\delta+1}\, dx\\
& = \mathcal O(\tau^{(\delta+1)\theta}) \quad\text{as}\quad\tau \to \infty.
\end{align*}
The other estimates follow readily.
\end{proof}

\subsection{Estimating $\|f\|_{L^2_{\delta+1}}$}\label{section on L^2 norm of f}

With our choice of $a$, we have
\begin{equation}\label{transport equation part}
2\rho\cdot \nabla a + (\nabla\beta^\sharp\cdot a) \rho + \nabla\alpha^\sharp \times (\rho \times a) = s_0 \tau^{-\theta} 2\rho\cdot(\nabla \chi)(\tau^{-\theta}\,\cdot\,) \Psi^\sharp e^{\alpha^\sharp / 2} e^{\chi_\tau \Psi^\sharp} \overline{\rho}.
\end{equation}
Then, as in the proof of \eqref{lemma tech 1}, we use \eqref{eqn::W^2,infty norms of alpha & beta sharps}, \eqref{estimate W^1,infty of Psi^sharp}, \eqref{estimate pointwise derivatives of Psi^sharp} and \eqref{integral Japanese symbol estimate} to obtain
\begin{multline*}
\| i\tau (2\rho\cdot \nabla a + (\nabla\beta^\sharp\cdot a) \rho + \nabla\alpha^\sharp \times (\rho \times a)) \|_{L^2_{\delta+1}} \\
\le \mathcal O(\tau^{1-\theta}) \Big(\int_{|x_T| \le \tau^\theta,\, |x_\perp| \le R} (1 + |x_T|^2)^{-1} (1 + |x|^2)^{\delta+1}\, dx \Big)^{1/2} = \mathcal O(\tau^{1+\delta\theta})\quad\text{as}\quad\tau \to \infty.
\end{multline*}
Using \eqref{sigma and mu are constants outside of a compact set},
\eqref{supports of derivatives of alpha and beta},
\eqref{L^infty norms of derivatives of alpha and beta errors},
\eqref{eqn::W^2,infty norms of alpha & beta sharps} and
\eqref{estimate W^1,infty of Psi^sharp}, it is then straightforward to show
that
\begin{align*}
\| i\tau ((\nabla(\beta-\beta^\sharp)\cdot a) \rho + \nabla(\alpha-\alpha^\sharp) \times (\rho \times a)) \|_{L^2_{\delta+1}} &= o(1)\ \quad\text{as}\quad\tau \to \infty,\\[0.05cm]
\| (\nabla\beta\cdot a) i\zeta_1 + \nabla\alpha \times (i\zeta_1 \times a) \|_{L^2_{\delta+1}} &= \mathcal O(1)\quad\text{as}\quad\tau \to \infty,\\
\|i\omega (\sigma\mu - \sigma_0 \mu_0) a\|_{L^2_{\delta+1}} &= \mathcal O(1)\quad\text{as}\quad\tau \to \infty.
\end{align*}
Now, using expressions \eqref{eqn::W^2,infty norms of alpha & beta sharps} and \eqref{estimate W^1,infty of Psi^sharp}, we find that
$$
\|\p_j a \|_{L^2_{\delta+1}} \le \mathcal O(1)\|\p_j\alpha^\sharp\|_{L^2_{\delta+1}} + \mathcal O(1)\|\p_j(\chi_\tau \Psi^\sharp)\|_{L^2_{\delta+1}}
$$
Thus with \eqref{supports of derivatives of alpha and beta}, \eqref{eqn::W^2,infty norms of alpha & beta sharps} and \eqref{lemma tech 1}, we obtain
\begin{equation}\label{L2 norm of der a}
   \|\p_j a \|_{L^2_{\delta+1}}
   = o(\tau^{\theta})\quad\text{as}\quad\tau \to \infty,
\end{equation}
and therefore,
$$
\| 2i \zeta_1\cdot \nabla a \|_{L^2_{\delta+1}} = o(\tau^{\theta})\quad\text{as}\quad\tau \to \infty,
$$
$$
\|\nabla\alpha\times \nabla\times a\|_{L^2_{\delta+1}} \le \mathcal O(1) \|\nabla\times a\|_{L^2_{\delta+1}} \le \mathcal O(1)\sum_{j=1}^3\|\p_j a \|_{L^2_{\delta+1}} = o(\tau^{\theta})\quad\text{as}\quad\tau \to \infty
$$
and
$$
\|\nabla(\nabla\beta\cdot a)\|_{L^2_{\delta+1}} \le \|(\nabla\nabla\beta) a\|_{L^2_{\delta+1}} + \mathcal O(1) \sum_{j=1}^3 \|\p_j a\|_{L^2_{\delta+1}} = o(\tau^\theta)\quad\text{as}\quad\tau \to \infty.
$$
In the last step, we used \eqref{sigma and mu are constants outside of a compact set} and that $\supp(\nabla\nabla\beta)$ is compact.
Finally, by \eqref{eqn::W^2,infty norms of alpha & beta sharps} and
\eqref{estimate W^1,infty of Psi^sharp}
\begin{multline*}
\|\p_j\p_k a\|_{L^2_{\delta+1}} \le \mathcal O(1) \|\p_j\p_k \alpha^\sharp\|_{L^2_{\delta+1}} + \mathcal O(1) \|\p_j\alpha^\sharp \p_j\alpha^\sharp\|_{L^2_{\delta+1}}\\
+ \mathcal O(1)\|\p_j\p_k(\chi_\tau \Psi^\sharp)\|_{L^2_{\delta+1}} + \mathcal O(1)\|\p_j(\chi_\tau \Psi^\sharp)\|_{L^2_{\delta+1}}\\
+ \mathcal O(1)\|\p_k(\chi_\tau \Psi^\sharp)\|_{L^2_{\delta+1}} + \mathcal O(1)\|\p_j(\chi_\tau \Psi^\sharp)\p_k(\chi_\tau \Psi^\sharp)\|_{L^2_{\delta+1}}
\end{multline*}
whence, by \eqref{supports of derivatives of alpha and beta}, \eqref{eqn::W^2,infty norms of alpha & beta sharps}, \eqref{lemma tech 1}, \eqref{lemma tech 2} and \eqref{lemma tech 4},
\begin{equation}\label{L2 norm of second der of a}
   \|\p_j\p_k a\|_{L^2_{\delta+1}}
   = o(\tau)\quad\text{as}\quad\tau \to \infty.
\end{equation}
Thus,
$$
\|\Delta a\|_{L^2_{\delta+1}} \le \sum_{j=1}^3 \|\p_j^2 a\|_{L^2_{\delta+1}} = o(\tau)\quad\text{as}\quad\tau \to \infty.
$$
Combining all of the above estimates, we come to
$$
\|f\|_{L^2_{\delta+1}} = o(\tau)\quad\text{as}\quad \tau\to\infty.
$$

\subsection{Estimating $\|\nabla_\zeta \times f\|_{L^2_{\delta+1}}$}
\label{section on L^2 norm of weighted curl of f}

Since
$$
\nabla_\zeta\times f = \nabla\times f + i\zeta\times f\quad\text{and}\quad \|i\zeta\times f\|_{L^2_{\delta+1}} = o(\tau^2),
$$
we need to estimate $\|\nabla\times f\|_{L^2_{\delta+1}}$. By straightforward calculations,
\begin{align*}
  \nabla\times f &= -\nabla\times \Delta a - \nabla\times(\nabla\alpha\times \nabla\times a)  - i\omega \nabla\times\big((\sigma\mu - \sigma_0 \mu_0) a \big)
  - 2i \nabla\times(\zeta_1\cdot \nabla a) - \nabla (\nabla\beta\cdot a)\times i\zeta_1 - \nabla\times\big(\nabla\alpha \times (i\zeta_1 \times a)\big)\\
& \quad - i\tau \nabla(\nabla(\beta-\beta^\sharp)\cdot a)\times \rho - i\tau\nabla\times\big(\nabla(\alpha-\alpha^\sharp) \times (\rho \times a)\big)
 - i\tau \nabla\times \big(2\rho\cdot \nabla a + (\nabla\beta^\sharp\cdot a) \rho + \nabla\alpha^\sharp \times (\rho \times a)\big).
\end{align*}
By \eqref{transport equation part},
$$
\nabla\times \big(2\rho\cdot \nabla a + (\nabla\beta^\sharp\cdot a) \rho + \nabla\alpha^\sharp \times (\rho \times a)\big) = \frac{s_0}{\tau^{\theta}} \nabla\Big(2\rho\cdot(\nabla \chi)(\tau^{-\theta}\,\cdot\,) \Psi^\sharp e^{\alpha^\sharp / 2} e^{\chi_\tau \Psi^\sharp}\Big) \times \overline{\rho}.
$$
Then, as in the proof of \eqref{lemma tech 4}, we use
\eqref{eqn::W^2,infty norms of alpha & beta sharps} and
\eqref{estimate W^1,infty of Psi^sharp},
\begin{align*}
\|i\tau \nabla &\times \big(2\rho\cdot \nabla a + (\nabla\beta^\sharp\cdot a) \rho + \nabla\alpha^\sharp \times (\rho \times a)\big)\|_{L^2_{\delta+1}}
\le \mathcal O(\tau^{1-\theta})\sum_{j,k=1}^3\|\p_j(\p_k\chi(\tau^{-\theta}\,\cdot\,)\Psi^\sharp e^{\alpha^\sharp / 2} e^{\chi_\tau \Psi^\sharp})\|_{L^2_{\delta+1}}\\
&\le \mathcal O(\tau^{1-2\theta})\sum_{j,k=1}^3\|\p_j\p_k\chi(\tau^{-\theta}\,\cdot\,)\Psi^\sharp\|_{L^2_{\delta+1}} + \mathcal O(\tau^{1-\theta})\sum_{j,k=1}^3\|\p_k\chi(\tau^{-\theta}\,\cdot\,)\p_j\Psi^\sharp\|_{L^2_{\delta+1}}
+ \mathcal O(\tau^{1-\theta})\sum_{j,k=1}^3\|\p_k\chi(\tau^{-\theta}\,\cdot\,)\p_j\alpha^\sharp \Psi^\sharp\|_{L^2_{\delta+1}} \\
&\quad+ \mathcal O(\tau^{1-2\theta})\sum_{j,k=1}^3\|\p_k\chi(\tau^{-\theta}\,\cdot\,)\p_j\chi(\tau^{-\theta}\,\cdot\,) \Psi^\sharp{}^2\|_{L^2_{\delta+1}}
+ \mathcal O(\tau^{1-\theta})\sum_{j,k=1}^3\|\p_k\chi(\tau^{-\theta}\,\cdot\,)\p_j\Psi^\sharp \chi_\tau \Psi^\sharp\|_{L^2_{\delta+1}}
= o(\tau)
   \quad\text{as}\quad \tau\to\infty,
\end{align*}
where in the last step, we also used \eqref{supports of derivatives of
  alpha and beta}, \eqref{estimate pointwise derivatives of
  Psi^sharp}, \eqref{integral Japanese symbol estimate} and
\eqref{integral Japanese symbol estimate -j}. Next, using \eqref{sigma
  and mu are constants outside of a compact set}, \eqref{supports of
  derivatives of alpha and beta}, \eqref{L^infty norms of derivatives
  of alpha and beta errors}, \eqref{eqn::W^2,infty norms of alpha &
  beta sharps}, \eqref{estimate W^1,infty of Psi^sharp} and \eqref{L2
  norm of der a}, we obtain
\begin{equation*}
\| i\tau \nabla(\nabla(\beta-\beta^\sharp)\cdot a)\times \rho\|_{L^2_{\delta+1}}\le \mathcal O(\tau)\sum_{j,k} \|\p_j\p_k (\beta - \beta^\sharp)a\|_{L^2_{\delta+1}} + \mathcal O(\tau)\sum_{j,k}\|\p_k (\beta - \beta^\sharp) \p_j a\|_{L^2_{\delta+1}}
= o(\tau^{1+\theta})\quad\text{as}\quad \tau\to\infty.
\end{equation*}
Similarly,
\begin{equation*}
\| i\tau \nabla\times\big(\nabla(\alpha-\alpha^\sharp) \times (\rho \times a)\big) \|_{L^2_{\delta+1}}\le \mathcal O(\tau)\sum_{j,k} \|\p_j\p_k (\alpha-\alpha^\sharp)a\|_{L^2_{\delta+1}} + \mathcal O(\tau)\sum_{j,k}\|\p_k (\alpha-\alpha^\sharp) \p_j a\|_{L^2_{\delta+1}} = o(\tau^{1+\theta}),
\end{equation*}
\begin{equation*}
\|\nabla (\nabla\beta\cdot a) \times i\zeta_1 \|_{L^2_{\delta+1}}\le \mathcal O(1)\sum_{j,k} \|\p_j\p_k \beta\, a\|_{L^2_{\delta+1}} + \mathcal O(1)\sum_{j,k}\|\p_k \beta\, \p_j a\|_{L^2_{\delta+1}} = o(\tau^{\theta}),
\end{equation*}
\begin{equation*}
\|\nabla\times\big(\nabla\alpha \times (i\zeta_1 \times a)\big) \|_{L^2_{\delta+1}}\le \mathcal O(1)\sum_{j,k} \|\p_j\p_k \alpha\, a\|_{L^2_{\delta+1}} + \mathcal O(1)\sum_{j,k}\|\p_k \alpha\, \p_j a\|_{L^2_{\delta+1}} = o(\tau^{\theta}),
\end{equation*}
and
\begin{equation*}
\|i\omega\nabla\times\big((\sigma\mu - \sigma_0 \mu_0) a\big)\|_{L^2_{\delta+1}}\le \mathcal O(1)\sum_{j} \|\p_j(\sigma\mu) a\|_{L^2_{\delta+1}} + \mathcal O(1)\sum_{j,k}\|(\sigma\mu - \sigma_0\mu_0) \p_j a\|_{L^2_{\delta+1}} = o(\tau^\theta)
\end{equation*}
as $\tau \to \infty$. Using \eqref{L2 norm of second der of a} we find that
$$
\|2i\nabla\times(\zeta_1\cdot \nabla a)\|_{L^2_{\delta+1}} \le \mathcal O(1) \sum_{j,k} \|\p_j\p_k a\|_{L^2_{\delta+1}} = o(\tau)
$$
and
$$
\|\nabla\times(\nabla\alpha\times \nabla\times a)\|_{L^2_{\delta+1}}\le \mathcal O(1)\sum_{j,k,l} \|\p_j\p_k \alpha\, \p_l a\|_{L^2_{\delta+1}} + \mathcal O(1)\sum_{j,k,l}\|\p_k \alpha\, \p_j\p_l a\|_{L^2_{\delta+1}} = o(\tau)
$$
as $\tau \to \infty$.

Using \eqref{eqn::W^2,infty norms of alpha & beta sharps} and \eqref{estimate W^1,infty of Psi^sharp}, we estimate
\begin{align*}
\|\p_l\p_j\p_k a\|_{L^2_{\delta+1}} &\le \mathcal O(1)\|\p_l\p_j\p_k \alpha^\sharp\|_{L^2_{\delta+1}} + \mathcal O(1)\|\p_l\p_j \alpha^\sharp \p_k \alpha^\sharp\|_{L^2_{\delta+1}}
+ \mathcal O(1)\|\p_j \alpha^\sharp \p_l\p_k \alpha^\sharp\|_{L^2_{\delta+1}}
 \\ &\
 + \mathcal O(1)\|\p_j\p_k \alpha^\sharp \p_l \alpha^\sharp\|_{L^2_{\delta+1}}
 + \mathcal O(1)\|\p_l \alpha^\sharp \p_j \alpha^\sharp \p_k \alpha^\sharp\|_{L^2_{\delta+1}}
 + \mathcal O(1) \|\p_l\p_j\p_k(\chi_\tau \Psi^\sharp)\|_{L^2_{\delta+1}}
 \\ &\
 + \mathcal O(1) \|\p_l\p_j\alpha^\sharp \p_k(\chi_\tau \Psi^\sharp)\|_{L^2_{\delta+1}}
 + \mathcal O(1) \|\p_l\p_j(\chi_\tau \Psi^\sharp) \p_k\alpha^\sharp\|_{L^2_{\delta+1}} + \mathcal O(1) \|\p_l\p_j(\chi_\tau \Psi^\sharp) \p_k(\chi_\tau \Psi^\sharp)\|_{L^2_{\delta+1}}
 \\ &\
 + \mathcal O(1) \|\p_j\alpha^\sharp \p_l\p_k(\chi_\tau \Psi^\sharp)\|_{L^2_{\delta+1}} + \mathcal O(1) \|\p_j(\chi_\tau \Psi^\sharp) \p_l\p_k\alpha^\sharp\|_{L^2_{\delta+1}} + \mathcal O(1) \|\p_j(\chi_\tau \Psi^\sharp) \p_l\p_k(\chi_\tau \Psi^\sharp)\|_{L^2_{\delta+1}}
 \\ &\
 + \mathcal O(1) \|\p_j\p_k\alpha^\sharp \p_l(\chi_\tau \Psi^\sharp)\|_{L^2_{\delta+1}} + \mathcal O(1) \|\p_j \p_k(\chi_\tau \Psi^\sharp) \p_l\alpha^\sharp\|_{L^2_{\delta+1}}+ \mathcal O(1) \|\p_j \p_k(\chi_\tau \Psi^\sharp) \p_l(\chi_\tau \Psi^\sharp)\|_{L^2_{\delta+1}}
 \\ &\
 + \mathcal O(1) \|\p_j\alpha^\sharp \p_k\alpha^\sharp \p_l(\chi_\tau \Psi^\sharp)\|_{L^2_{\delta+1}}
 + \mathcal O(1) \|\p_j\alpha^\sharp \p_k(\chi_\tau \Psi^\sharp) \p_l\alpha^\sharp\|_{L^2_{\delta+1}}
 + \mathcal O(1) \|\p_j\alpha^\sharp \p_k(\chi_\tau \Psi^\sharp) \p_l(\chi_\tau \Psi^\sharp)\|_{L^2_{\delta+1}}
 \\ &\
 + \mathcal O(1) \|\p_j(\chi_\tau \Psi^\sharp) \p_k\alpha^\sharp \p_l\alpha^\sharp\|_{L^2_{\delta+1}}
 + \mathcal O(1) \|\p_j(\chi_\tau \Psi^\sharp) \p_k\alpha^\sharp \p_l(\chi_\tau \Psi^\sharp)\|_{L^2_{\delta+1}}
 + \mathcal O(1) \|\p_j(\chi_\tau \Psi^\sharp) \p_k(\chi_\tau \Psi^\sharp) \p_l\alpha^\sharp\|_{L^2_{\delta+1}}
 \\ &\hspace*{4.55cm}
 + \mathcal O(1) \|\p_j(\chi_\tau \Psi^\sharp) \p_k(\chi_\tau \Psi^\sharp) \p_l(\chi_\tau \Psi^\sharp)\|_{L^2_{\delta+1}}\quad\text{as}\quad\tau \to \infty
\end{align*}
and conclude that
\[
   \|\p_l\p_j\p_k a\|_{L^2_{\delta+1}}
       = o(\tau^{1+\epsilon})\quad\text{as}\quad\tau \to \infty.\ \
\]
This implies that
$$
\|\nabla\times \Delta a\|_{L^2_{\delta+1}} = o(\tau^{1+\epsilon})\quad\text{as}\quad\tau \to \infty.
$$
Combining all of these, we finally come to $\|\nabla\times
f\|_{L^2_{\delta+1}} = o(\tau^{1+\epsilon})$ and, hence,
$\|\nabla_\zeta\times f\|_{L^2_{\delta+1}} =
o(\tau^2)\quad\text{as}\quad\tau \to \infty$.

\subsection{Construction of complex geometric optics solutions} Now, we are ready to construct complex geometric optics solutions for the system \eqref{eqn3-1} and \eqref{eqn3-2} which is equivalent to
\begin{align}
e^{-i \zeta\cdot x}L_{\sigma,\mu}(e^{i \zeta\cdot x}r) & = - f, \label{eqn LE=0 rewritten}\\
\nabla_\zeta \cdot r + \nabla\log\mu \cdot r & = - \nabla_\zeta \cdot a + \nabla\log\mu \cdot a. \label{eqn div(sigma E)=0 rewritten}
\end{align}
According to the discussion in Section~\ref{section on L^2 norm of f}, we have $\|f\|_{L^2_{\delta+1}} = o(\tau)$ as $\tau \to \infty$.\smallskip

First, we need to show that for sufficiently large $|\zeta|$ there is $r\in L^2_\delta$ solving \eqref{eqn LE=0 rewritten}. Using \eqref{eqn:: conj grad and laplacian},
\begin{equation}\label{another pde for r}
e^{-i \zeta\cdot x} L_{\sigma,\mu}(e^{i \zeta\cdot x}r) = -\Delta_\zeta r - \nabla_\zeta(\nabla\log\mu\cdot r) - \nabla\log\sigma\times \nabla_\zeta\times r  - i\omega (\sigma\mu - \sigma_0 \mu_0) r.
\end{equation}
We have
$$
\nabla_\zeta(\nabla\log\mu\cdot r) = \nabla\log\mu \times
\nabla_\zeta\times r + \nabla\log\mu \cdot \nabla_\zeta r +
\nabla\nabla(\log\mu)r
$$
Therefore, \eqref{eqn LE=0 rewritten} can be written as
\begin{equation}\label{pre final pde for r}
-\Delta_\zeta r - \nabla\log\mu \cdot \nabla_\zeta r - \nabla\log(\sigma\mu) \times \nabla_\zeta\times r - V_1 r = - f,
\end{equation}
where
$$
V_1 := \nabla\nabla(\log\mu) + i\omega (\sigma\mu - \sigma_0 \mu_0).
$$
We need to deal with the third term on the left hand-side of
\eqref{pre final pde for r}. We define $Q := \nabla_\zeta\times r$, and
find that
$$
\Delta_\zeta r = \nabla_\zeta \nabla_\zeta \cdot r - \nabla_\zeta \times \nabla_\zeta \times r.
$$
Also, it follows from \eqref{eqn div(sigma E)=0 rewritten} that
$$
\nabla_\zeta \nabla_\zeta \cdot r = - \nabla_\zeta \nabla_\zeta \cdot a - \nabla_\zeta(\nabla\log\mu \cdot a) - \nabla_\zeta(\nabla\log\mu \cdot r).
$$
Substituting these into \eqref{another pde for r}, we come to
$$
e^{-i \zeta\cdot x} L_{\sigma,\mu}(e^{i \zeta\cdot x}r) = \nabla_\zeta \times \nabla_\zeta \times r + \nabla_\zeta \nabla_\zeta \cdot a + \nabla_\zeta(\nabla\log\mu \cdot a) - \nabla\log\sigma\times \nabla_\zeta\times r  - i\omega (\sigma\mu - \sigma_0 \mu_0) r.
$$
Hence, \eqref{eqn LE=0 rewritten} implies
$$
\nabla_\zeta \times Q + \nabla_\zeta \nabla_\zeta \cdot a + \nabla_\zeta(\nabla\log\mu \cdot a) - \nabla\log\sigma\times Q  - i\omega (\sigma\mu - \sigma_0 \mu_0) r = - f.
$$
Applying $\nabla_\zeta\times$ to it, we get
$$
\nabla_\zeta \times \nabla_\zeta \times Q - \nabla_\zeta\times (\nabla\log\sigma\times Q) - i\omega \nabla(\sigma \mu) \times r - i\omega (\sigma \mu - \sigma_0 \mu_0) Q = - \nabla_\zeta\times f
$$
as
$$
\nabla_\zeta\times (\nabla_\zeta \nabla_\zeta \cdot a) = 0\quad\text{and}\quad \nabla_\zeta\times (\nabla_\zeta(\nabla\log\mu \cdot a) )= 0.
$$
Using
the fact that $\nabla_\zeta \cdot Q = 0$, we can write
$$
\nabla_\zeta\times (\nabla\log\sigma\times Q) = \nabla\nabla(\log\sigma) Q - (\Delta\log\sigma)Q - \nabla\log\sigma \cdot \nabla_\zeta Q.
$$
Thus, we come to
\begin{equation}\label{final eqn for Q}
-\Delta_\zeta Q - \nabla\log\sigma \cdot \nabla_\zeta Q - V_2 Q \\
= i\omega \nabla(\sigma\mu)\times r - \nabla_\zeta \times f,
\end{equation}
where
$$
V_2:= \nabla\nabla(\log\sigma) + i\omega(\sigma\mu - \sigma_0 \mu_0) - \Delta\log\sigma.
$$
According to the discussion in Section~\ref{section on L^2 norm of weighted curl of f}, for sufficiently large $\tau$, there are bounded inverses
$$
G_{\zeta, \mu}: L^2_{\delta + 1} \to L^2_\delta \quad\text{and}\quad G_{\zeta, \sigma}: L^2_{\delta + 1} \to L^2_\delta
$$
of $-\Delta_\zeta - \nabla\log\mu \cdot \nabla_\zeta$ and $-\Delta_\zeta - \nabla\log\sigma \cdot \nabla_\zeta$, respectively, satisfying
\begin{equation}\label{ineq::operator norms of G sigma}
\|G_{\zeta,\mu}\|_{L^2_{\delta + 1} ; L^2_\delta} \le \mathcal O\Big(\frac{1}{\tau}\Big)\quad\text{as}\quad\tau \to \infty\end{equation}
and
\begin{equation}\label{ineq::operator norms of G mu}
\|G_{\zeta,\sigma}\|_{L^2_{\delta + 1} ; L^2_\delta} \le \mathcal O\Big(\frac{1}{\tau}\Big)\quad\text{as}\quad\tau \to \infty.\end{equation}
Furthermore, $G_{\zeta,\mu}$ and $G_{\zeta,\sigma}$ map the space $L^2_{\delta + 1}$ into $H^1_\delta$.\smallskip

If $\tau$ is large enough, we apply $G_{\zeta, \sigma}$ to \eqref{final eqn for Q} and obtain the following identity
\begin{equation}\label{int identity for Q}
(\id - G_{\zeta, \sigma}V_2) Q = G_{\zeta, \sigma}\big\{i\omega\nabla(\sigma\mu)\times r - \nabla_\zeta \times f \big\}.
\end{equation}
Since the support of $V_2$ is compact in $\R^3$, the operator $\id - G_{\zeta, \sigma}V_2$ is invertible in $L^2_\delta$ for large enough $\tau$. Also, if $r\in L^2_
\delta$, then $\nabla(\sigma\mu)\times r\in L^2_{\delta+1}$ by \eqref{sigma and mu are constants outside of a compact set}. Thus, one can solve the above identity for $Q\in L^2_\delta$,
\begin{equation}\label{solution for Q}
Q = (\id - G_{\zeta, \sigma} V_2)^{-1} G_{\zeta, \sigma}\big\{i\omega\nabla(\sigma\mu)\times r - \nabla_\zeta \times f \big\}.
\end{equation}
Substituting this solution into \eqref{pre final pde for r}, we obtain
\begin{equation}\label{final integro pde for r}
-\Delta_\zeta r - \nabla\log\mu \cdot \nabla_\zeta r - i\omega W r - V_1 r = - F,
\end{equation}
where
\begin{align*}
W &:=  \nabla\log(\sigma\mu) \times (\id - G_{\zeta, \sigma} V_1)^{-1}\circ G_{\zeta, \sigma}\circ \nabla(\sigma\mu)\times\ ,\\
F &:= f + \nabla\log(\sigma\mu) \times (\id - G_{\zeta, \sigma} V_1)^{-1}\circ G_{\zeta, \sigma}\circ \nabla_\zeta \times f.
\end{align*}
It follows from \eqref{ineq::operator norms of G mu} and \eqref{sigma and mu are constants outside of a compact set} that
\begin{equation}\label{operator norm of W}
\|W\|_{L^2_{\delta} ; L^2_{\delta}} = \mathcal O\Big(\frac{1}{\tau}\Big)\quad\text{as}\quad\tau \to \infty.
\end{equation}
and
$$
 \|F\|_{L^2_{\delta+1}} \le  \|f\|_{L^2_{\delta+1}} + \mathcal O\Big(\frac{1}{\tau}\Big)\| \nabla_\zeta \times f\|_{L^2_{\delta+1}} = o(\tau)\quad\text{as}\quad\tau \to \infty.
$$
Applying $G_{\zeta, \mu}$ to \eqref{final integro pde for r}, we get
\begin{equation}\label{int identity for r}
(\id - i\omega G_{\zeta, \mu} W - G_{\zeta, \mu} V_1) r = - G_{\zeta, \mu} F.
\end{equation}
Since $V_1$ is compactly supported and $W$ satisfies \eqref{operator norm of W}, the operator $\id - i\omega G_{\zeta, \mu} W - G_{\zeta, \mu} V_1$ is invertible in $L^2_\delta$ for $\tau$ sufficiently large. Therefore, one can solve the above identity for $r\in L^2_\delta$ by
$$
r = - (\id - i\omega G_{\zeta, \mu} W - G_{\zeta, \mu} V_1)^{-1} G_{\zeta, \mu}F.
$$
Finally, by \eqref{ineq::operator norms of G mu} and \eqref{solution for Q}, we can show that
$$
\|r\|_{L^2_{\delta}} \le \mathcal O\Big(\frac{1}{\tau}\Big)\|F\|_{L^2_{\delta+1}} = o(1)\quad\text{as}\quad\tau \to \infty
$$
and
$$
\|\nabla_\zeta\times r\|_{L^2_{\delta}} \le \mathcal O\Big(\frac{1}{\tau}\Big) \Big\{\|r\|_{L^2_{\delta}} + \|\nabla_\zeta\times f\|_{L^2_{\delta+1}}\Big\} = o(\tau)\quad\text{as}\quad\tau \to \infty. 
$$
It follows from \eqref{int identity for Q} and \eqref{int identity for r} that
\begin{align*}
r &= i\omega G_{\zeta, \mu} (Wr) + G_{\zeta, \mu} (V_1 r) - G_{\zeta, \mu} F,\\
Q &= G_{\zeta, \sigma}(V_2 Q) + G_{\zeta, \sigma}\big\{i\omega\nabla(\sigma\mu)\times r - \nabla_\zeta \times f \big\}.
\end{align*}
Since $V_1$, $V_2$ and $W$ are compactly supported and $G_{\zeta,\mu}$ and $G_{\zeta,\sigma}$ map the space $L^2_{\delta + 1}$ into $H^1_\delta$, this implies that $r,\nabla_\zeta\times r \in H^1_\delta$. Thus, we have constructed the following complex geometric optics solution for \eqref{eqn3-1}
$$
H(x;\zeta) = e^{i\zeta\cdot x}\big\{e^{-\alpha^\sharp(x;\tau)/2}\rho + s_0 b e^{\alpha^\sharp(x;\tau)/2} e^{\Psi^\sharp(x;\rho)}\overline\rho + r(x;\zeta)\big\}.
$$
Our next step is to show that $\nabla\cdot(\mu H)=0$.
We observe that \eqref{eqn3-1} is equivalent to
$$
\nabla\times(\sigma^{-1}\nabla\times H) - \sigma^{-1}\nabla(\mu^{-1}\nabla\cdot(\mu H)) - i\omega\mu H = 0.
$$
Applying the divergence to this identity and setting $v = \mu^{-1}\nabla\cdot(\mu H)$, we get
$$
- \nabla\cdot(\sigma^{-1}\nabla v) - i\omega\mu v = 0
$$
which can be written as
$$
- \Delta \tilde v + \tilde q \tilde v = 0,\quad \tilde v := \sigma^{-1/2} v,\quad \tilde q := \sigma^{1/2} \Delta \sigma^{-1/2} - i\omega \sigma \mu.
$$
Straightforward calculations give
$$
\tilde v = e^{i\zeta\cdot x} u,\quad u:= \sigma^{-1/2}\mu^{-1}\nabla_\zeta\cdot\big(\sigma(a+r)\big).
$$
Then, by \eqref{eqn:: conj grad and laplacian}, $u$ satisfies
$$
-\Delta_\zeta u= q u,\quad q := - \sigma^{1/2} \Delta \sigma^{-1/2} + i\omega ( \sigma \mu - \sigma_0 \mu_0 ).
$$
By \cite[Theorem~1.6]{sylvester1987global}, again, there is a bounded inverse, $G_\zeta : L^2_{\delta+1} \to L^2_{\delta}$ of $-\Delta_\zeta$ such that $\|G_\zeta\|_{L^2_{\delta+1} ; L^2_{\delta}} \le \mathcal O(|\zeta|^{-1})$ as $|\zeta|\to\infty$. Since $u = G_\zeta (qu)$ for large enough $|\zeta|$ and $q$ is compactly supported,
$$
\|u\|_{L^2_{\delta}} \le \|G_\zeta(qu)\|_{L^2_{\delta}} \le \mathcal O\Big(\frac{1}{|\zeta|}\Big)\|qu\|_{L^2_{\delta+1}} \le \mathcal O\Big(\frac{1}{|\zeta|}\Big)\|u\|_{L^2_{\delta}}\quad\text{as}\quad |\zeta|\to\infty.
$$
This implies that $u=0$ and hence $\nabla\cdot(\mu H)=0$ for sufficiently large $|\zeta|$. 
Thus, restricting $H$ onto $\Omega$, we get $H\in H^1(\Omega;\C^3)$ solving \eqref{eqn::curl-curl equation} and satisfying $\nabla\times H\in H^1(\Omega;\C^3)$. Clearly, $\nu\times H|_{\p\Omega}\in H^{1/2}(\p\Omega;\C^3)$. Then by \cite[Corollary~A.20.]{kirsch2016mathematical},
$$
\Div(\nu\times H|_{\p\Omega}) = - \nu\cdot (\nabla\times H)|_{\p\Omega} \in H^{1/2}(\p\Omega;\C^3).
$$
Therefore, $\nu\times H|_{\p\Omega}\in TH^{1/2}_{\Div}(\p\Omega)$ and hence $H\in H^1_{\Div}(\Omega)$. In a similar way, also using the fact that $H$ is a solution of \eqref{eqn::curl-curl equation}, one can easily show that $\nu\times (\nabla\times H)|_{\p\Omega} \in TH^{1/2}_{\Div}(\p\Omega)$ and hence $\nabla \times H\in H^1_{\Div}(\Omega)$.
Thus, we proved

\begin{Proposition}\label{prop CGO}
Let $\Omega\subset \R^3$ be an open bounded set with $C^{1,1}$ boundary and $\sigma,\mu\in C^2(\overline\Omega)$ with $\sigma \ge \sigma_0$, $\mu \ge \mu_0$ for some $\sigma_0,\mu_0>0$. Assume that $\sigma$ and $\mu$ can be extended positively to $\R^3$ so that $\sigma - \sigma_0, \mu - \mu_0 \in C^2_0(\R^3)$. Let $\zeta\in \C^3$ be such that $\zeta\cdot\zeta=i\omega\sigma_0\mu_0$, $\zeta = \tau \rho + \zeta_1$ where $\tau > 0$ is a large parameter, $\rho \in \C^3$ is independent of $\tau$ and satisfies $\Re\rho \cdot \Im\rho = 0$ and $|\Re\rho| = |\Im\rho| = 1$, and $\zeta_1 = \mathcal O(1)$ as $\tau \to \infty$. Then, for any $s_0 \in \R$, there is a solution $H\in H^1_{\Div}(\Omega)$ for \eqref{eqn::curl-curl equation} of the form
$$
H(x;\zeta) = e^{i\zeta\cdot x}\big\{e^{-\alpha^\sharp(x;\tau)/2}\rho + s_0 b e^{\alpha^\sharp(x;\tau)/2} e^{\Psi^\sharp(x;\rho)}\overline\rho + r(x;\zeta)\big\}.
$$
Furthermore, $\nabla\times H\in H^1_{\Div}(\Omega)$. The function
$\Psi^\sharp(\cdot, \rho; \tau)\in C^\infty(\R^3)$ satisfies
$\|\Psi^\sharp\|_{W^{1,\infty}(\R^3)} = \mathcal O(1)$ as
$\tau\to\infty$ and converges to $\Psi(\cdot, \rho) := - N_\rho^{-1}
\{\overline{\rho} \cdot \nabla\log(\sigma\mu)^{1/2}\} \in
L^\infty(\R^3)$ in $L^2_{\rm loc}(\R^3)$ as $\tau\to\infty$. The
function $\alpha^\sharp(\,\cdot\,; \tau)\in C^\infty(\R^3)$ satisfies
$\| \alpha^\sharp \|_{W^{2,\infty}(\R^3)} = \mathcal O(1)$ as
$\tau\to\infty$ and $\| \alpha^\sharp - \log\sigma
\|_{W^{2,\infty}(\R^3)} = o(1)$ as $\tau\to\infty$. The correction
term, $r$, satisfies $\|r\|_{L^2(\Omega;\C^3)} = o(1)$ and
$\|\nabla_\zeta\times r\|_{L^2(\Omega;\C^3)} = o(\tau)$ as
$\tau\to\infty$.
\end{Proposition}

\section{Proof of Theorem~\ref{main thm}}\label{section::proof of main thm}

Since we assume that $\p^\alpha \sigma_1|_{\p\Omega} = \p^\alpha
\sigma_2|_{\p\Omega}$ and $\p^\alpha \mu_1|_{\p\Omega} = \p^\alpha
\mu_2|_{\p\Omega}$ for $|\alpha| \le 2$, we can extend $\sigma_j$ and
$\mu_j$, $j=1,2$, to $C^2$ functions defined on $\R^3$, still denoted
by $\sigma_j$ and $\mu_j$, such that $\sigma_j \ge \sigma_0$, $\mu_j
\ge \mu_0$ on $\R^3$, $\sigma_j - \sigma_0, \mu_j - \mu_0 \in
C^2_0(\R^3)$ and $\sigma_1=\sigma_2$ and $\mu_1=\mu_2$ on
$\R^3\setminus\overline\Omega$. These kind of extensions (of Whitney
type) hold for all functions defined on any closed subset of $\R^3$
that can be approximated by certain polynomials. The argument to prove
the existence of such polynomials is similar to the one in
\cite[Section~2]{caro2013stability} for $C^{1,\varepsilon}$ functions
on $\overline\Omega$. The only difference, here, is that the authors
of \cite{caro2013stability} refer to \cite[Section~2 of
  Chapter~VI]{stein1970singular}, while we refer to \cite[Section~4.7
  of Chapter~VI]{stein1970singular}; see also
\cite[Section~3]{caro2014global}.

\begin{Proposition}
Let $\Omega\subset \R^3$ be a bounded domain with $C^{1,1}$ boundary and $\sigma_j,\mu_j\in C^2(\overline\Omega)$, $j=1,2$, with $\sigma_j \ge \sigma_0$, $\mu_j \ge \mu_0$ for some $\sigma_0,\mu_0>0$. Suppose that $Z_{\sigma_1,\mu_1}^\omega = Z_{\sigma_2,\mu_2}^\omega$; then
\begin{equation}\label{main integral identity}
\int_\Omega (\mu_1 - \mu_2) H_1\cdot H_2\,dx + \frac{1}{i\omega} \int_\Omega \frac{(\sigma_1 - \sigma_2)}{\sigma_1\sigma_2} \nabla\times H_1\cdot \nabla\times H_2\,dx = 0
\end{equation}
for all $H_j\in H_{\Div}^1(\Omega)$ with $\nabla\times H_j\in H_{\Div}^1(\Omega)$ solving
$$
\nabla\times(\sigma_j^{-1}\nabla\times H_j) - i\omega \mu_j H_j = 0\quad\text{in}\quad\Omega,\quad j=1,2.
$$
\end{Proposition}
\begin{proof}
Define
$$
E_j := \sigma_j^{-1} \nabla\times H_j,\quad j=1,2.
$$
Then $E_j\in H_{\Div}^1(\Omega)$ and $\nabla\times E_j = i\omega\mu_j H_j$. Hence $(H_j, E_j) \in H^1_{\Div}(\Omega) \times H^1_{\Div}(\Omega)$, $j=1,2$, solve
$$
\nabla\times E_j=i\omega\mu_j H_j\quad\text{and}\quad \nabla\times H_j=\sigma_j E_j \quad\text{in}\quad\Omega,\quad j=1,2.
$$
Then the assumption $Z_{\sigma_1,\mu_1}^\omega = Z_{\sigma_2,\mu_2}^\omega$ implies existence of $(H', E')\in H^1_{\Div}(\Omega) \times H^1_{\Div}(\Omega)$ satisfying
$$
\nabla\times E'=i\omega\mu_2 H'\quad\text{and}\quad \nabla\times H'=\sigma_2 E'\quad\text{in}\quad\Omega
$$
and
$$
\bt(H') = \bt(H_1)\quad\text{and}\quad \bt(E') = \bt(E_1)\quad\text{on}\quad\p \Omega.
$$
Integrating by parts,
\begin{align*}
\int_\Omega \nabla\times(H'-H_1)\cdot E_2\,dx - \int_\Omega i\omega\mu_2 (H'-H_1)\cdot H_2\,dx&= \int_\Omega \nabla\times(H'-H_1)\cdot E_2\,dx - \int_\Omega (H'-H_1)\cdot \nabla\times E_2\,dx\\
&= \int_{\p\Omega} \bt (H'-H_1)\cdot E_2\,dS(x) = 0,
\end{align*}
where $dS$ is the surface measure on $\p \Omega$. Similarly,
$$
\int_\Omega \nabla\times(E'-E_1)\cdot H_2\,dx - \int_\Omega \sigma_2 (E'-E_1)\cdot E_2\,dx = 0.
$$
Adding these two identities, we obtain
$$
\int_\Omega [\nabla\times(H'-H_1) - \sigma_2 (E'-E_1)]\cdot E_2\,dx + \int_\Omega [\nabla\times(E'-E_1) - i\omega\mu_2 (H'-H_1)]\cdot H_2\,dx = 0.
$$
It is easy to show that
$$
\nabla\times(H'-H_1) - \sigma_2 (E'-E_1) = (\sigma_2 - \sigma_1) E_1,\quad \nabla\times(E'-E_1) - i\omega\mu_2 (H'-H_1) = i\omega(\mu_2 - \mu_1) H_1.
$$
Substituting these into the latter integral identity, we come to
$$
\int_\Omega (\sigma_2 - \sigma_1) E_1\cdot E_2\,dx + \int_\Omega i\omega(\mu_2 - \mu_1) H_1\cdot H_2\,dx = 0.
$$
This implies \eqref{main integral identity}.
\end{proof}

Let $\xi,\rho_1,\rho_2\in \R^3$ be such that $|\rho_1| = |\rho_2| = 1$ and $\rho_1\cdot\rho_2 = \rho_1 \cdot \xi = \rho_2 \cdot \xi =0$. Consider
\begin{align*}
\zeta^1 &= \phantom{-}\frac{\xi}{2} + i \tau \rho_2 + \tau \sqrt{1 - \frac{|\xi|^2}{4\tau^2} + \frac{i\omega\sigma_0\mu_0}{\tau^2}} \rho_1,\\
\zeta^2 &= -\frac{\xi}{2} + i \tau \rho_2 + \tau \sqrt{1 - \frac{|\xi|^2}{4\tau^2} + \frac{i\omega\sigma_0\mu_0}{\tau^2}} \rho_1.
\end{align*}
Here, by $\sqrt{\, \cdot \,}$ we mean its principal branch. Then
$$
  \zeta^j = \tau \rho + \zeta^j_1\quad\text{with}\quad \zeta^j_1 = \mathcal O(1)\quad\text{as}\ \tau \to \infty\quad\text{and}\quad
  \zeta^1 - \zeta^2 = \xi,\quad \zeta^j\cdot\zeta^j = i\omega\sigma_0\mu_0,
$$
where $\rho:=\rho_1 + i\rho_2$. By Proposition~\ref{prop CGO}, there are complex geometric optics solutions $H_1,H_2\in H^1_{\Div}(\Omega)$, with $\nabla\times H_1, \nabla\times H_2 \in H^1_{\Div}(\Omega)$ satisfying
$$
\nabla\times(\sigma_1^{-1}\nabla\times H_1) - i\omega \mu_1 H_1 = 0\quad\text{and}\quad \nabla\times(\sigma_2^{-1}\nabla\times H_2) - i\omega \mu_2 H_2 = 0 \quad\text{in}\quad\Omega,
$$
respectively, which have the following forms
$$
H_1(x;\zeta^1) = e^{i\zeta^1\cdot x} \Big(a_1 \rho + \frac{1}{2} b_1 \overline\rho + r_1\Big),\quad H_2(x;\zeta^2) = e^{-i\zeta^2\cdot x} \Big(-a_2 \rho - \frac{1}{2} b_2 \overline\rho + r_2\Big),
$$
where
\begin{align*}
a_1 &= e^{-\alpha_1^\sharp(x;\tau)/2},\quad b_1 = e^{\alpha_1^\sharp(x;\tau)/2} e^{\Psi^\sharp_1(x, \rho; \tau)},\\
a_2 &= e^{-\alpha_2^\sharp(x;\tau)/2},\quad b_2 = e^{\alpha_2^\sharp(x;\tau)/2} e^{\Psi^\sharp_2(x, \rho; \tau)}.
\end{align*}
The functions $\Psi_1^\sharp(\cdot, \rho; \tau),\Psi_2^\sharp(\cdot, \rho; \tau)\in C^\infty(\R^3)$ satisfy
\begin{align}
\|\Psi_1^\sharp\|_{W^{1,\infty}(\R^3)}&,\ \|\Psi_2^\sharp\|_{W^{1,\infty}(\R^3)} = \mathcal O(1) \quad \text{as}\quad\tau\to\infty,\label{eqn 4.2}\\
\|\Psi_1^\sharp - \Psi_1\|_{L^2_{\rm loc}(\R^3)}&,\ \|\Psi_2^\sharp - \Psi_2\|_{L^2_{\rm loc}(\R^3)} = o(1) \quad\text{as}\quad\tau\to\infty,\label{eqn 4.3}
\end{align}
where
$$
\Psi_j(\cdot, \rho) := - N_\rho^{-1} \{\overline{\rho} \cdot \nabla\log(\sigma_j\mu_j)^{1/2}\} \in L^\infty(\R^3),\quad j=1,2.
$$
Furthermore, the functions $\alpha_1^\sharp(\,\cdot\,;\tau),\ \alpha_2^\sharp(\,\cdot\,;\tau) \in C^\infty(\R^3)$ satisfy
\begin{align}
\| \alpha_1^\sharp \|_{W^{2,\infty}(\R^3)}&,\ \| \alpha_2^\sharp \|_{W^{2,\infty}(\R^3)} = \mathcal O(1) \quad\text{as}\quad\tau\to\infty,\label{eqn 4.4}\\
\| \alpha_1^\sharp - \log\sigma_1 \|_{W^{2,\infty}(\R^3)}&,\ \| \alpha_2^\sharp - \log\sigma_2 \|_{W^{2,\infty}(\R^3)} = o(1) \quad\text{as}\quad\tau\to\infty.\label{eqn 4.5}
\end{align}
The correction terms, $r_1$, $r_2 \in H_{\Div}^1(\Omega)$, satisfy
\begin{equation}\label{eqn 4.6}
\|r_j\|_{L^2(\Omega)} = o(1)\quad\text{and}\quad\|\nabla_{\zeta^j}\times r_j\|_{L^2(\Omega)} = o(\tau) \quad \text{as}\quad \tau\to\infty,\quad j=1,2.
\end{equation}
Then, using that $\zeta^j = \tau\rho + \zeta^j_1$, we find that
\begin{align*}
\nabla\times E_1 &= e^{i\zeta^1\cdot x}\,\nabla_{\zeta^1}\times\Big(a_1 \rho + \frac{1}{2} b_1 \overline\rho + r_1\Big)
= e^{i\zeta^1\cdot x} \Big(\nabla_{\zeta^1_1}a_1\times \rho+\frac{1}{2}\nabla_{\zeta^1_1}b_1\times \overline\rho + b_1 \tau \rho_1\times \rho_2 + \nabla_{\zeta^1}\times r_1\Big),\\
\nabla\times E_2 &= e^{-i\zeta^2\cdot x} \,\nabla_{-\zeta^2}\times\Big(- a_2 \rho - \frac{1}{2} b_2 \overline\rho + r_2\Big)
= e^{-i\zeta^2\cdot x} \Big(- \nabla_{-\zeta^2_1}a_2\times \rho - \frac{1}{2}\nabla_{-\zeta^2_1}b_2\times \overline\rho + b_2 \tau \rho_1\times \rho_2 + \nabla_{-\zeta^2}\times r_2\Big).
\end{align*}
Substituting $H_1$, $H_2$, $\nabla\times H_1$ and $\nabla\times H_2$ into \eqref{main integral identity} and dividing the whole identity by $\tau^2$, we obtain
\begin{equation}\label{integral identity to find mu}
\begin{aligned}
\frac{1}{\tau^2}\int_\Omega &(\mu_1 - \mu_2) e^{i\xi\cdot x}\Big(a_1\rho + \frac{1}{2} b_1 \overline\rho + r_1\Big)\cdot \Big(- a_2\rho - \frac{1}{2} b_2 \overline\rho + r_2\Big)\,dx\\
&- \frac{1}{\tau^2}\int_\Omega \frac{1}{i\omega}\frac{(\sigma_1 - \sigma_2)}{\sigma_1\sigma_2} e^{i\xi\cdot x}\Big(\nabla_{\zeta^1_1}a_1\times \rho+\frac{1}{2}\nabla_{\zeta^1_1}b_1\times \overline\rho\Big)\cdot \Big(\nabla_{-\zeta^2_1}a_2\times \rho+\frac{1}{2}\nabla_{-\zeta^2_1}b_2\times \overline\rho\Big)\,dx\\
&+ \frac{1}{\tau}\int_\Omega \frac{1}{i\omega}\frac{(\sigma_1 - \sigma_2)}{\sigma_1\sigma_2} e^{i\xi\cdot x}\Big(\nabla_{\zeta^1_1}a_1\times \rho+\frac{1}{2}\nabla_{\zeta^1_1}b_1\times \overline\rho\Big)\cdot (b_2 \rho_1\times\rho_2)\,dx\\
&- \frac{1}{\tau}\int_\Omega \frac{1}{i\omega}\frac{(\sigma_1 - \sigma_2)}{\sigma_1\sigma_2} e^{i\xi\cdot x} (b_1 \rho_1\times \rho_2) \cdot \Big(\nabla_{-\zeta^2_1}a_2\times \rho+\frac{1}{2}\nabla_{-\zeta^2_1}b_2\times \overline\rho\Big)\,dx\\
&+ \frac{1}{\tau^2}\int_\Omega \frac{1}{i\omega}\frac{(\sigma_1 - \sigma_2)}{\sigma_1\sigma_2} e^{i\xi\cdot x}\Big(\nabla_{\zeta^1_1}a_1\times \rho+\frac{1}{2}\nabla_{\zeta^1_1}b_1\times \overline\rho\Big)\cdot (\nabla_{-\zeta^2}\times r_2)\,dx\\
&- \frac{1}{\tau^2}\int_\Omega \frac{1}{i\omega}\frac{(\sigma_1 - \sigma_2)}{\sigma_1\sigma_2} e^{i\xi\cdot x} (\nabla_{\zeta^1}\times r_1) \cdot \Big(\nabla_{-\zeta^2_1}a_2\times \rho+\frac{1}{2}\nabla_{-\zeta^2_1}b_2\times \overline\rho\Big)\,dx\\
&+ \frac{1}{\tau}\int_\Omega \frac{1}{i\omega}\frac{(\sigma_1 - \sigma_2)}{\sigma_1\sigma_2} e^{i\xi\cdot x} (b_1 \rho_1\times \rho_2) \cdot (\nabla_{-\zeta^2}\times r_2)\,dx\\
&+ \frac{1}{\tau}\int_\Omega \frac{1}{i\omega}\frac{(\sigma_1 - \sigma_2)}{\sigma_1\sigma_2} e^{i\xi\cdot x} (\nabla_{\zeta^1}\times r_1) \cdot (b_2 \rho_1\times\rho_2)\,dx + \int_\Omega \frac{1}{i\omega}\frac{(\sigma_1 - \sigma_2)}{\sigma_1\sigma_2} e^{i\xi\cdot x} b_1 b_2 \,dx = 0.
\end{aligned}
\end{equation}
For the last term on the left-hand side, we also used the property $|\rho_1\times\rho_2| = 1$ as $\rho_1 \cdot \rho_2 = 0$ and $|\rho_1| = |\rho_2| = 1$. By the Cauchy-Schwartz inequality,
\begin{multline*}
\Big|\int_\Omega (\mu_1 - \mu_2) e^{i\xi\cdot x}\Big(a_1\rho + \frac{1}{2} b_1 \overline\rho + r_1\Big)\cdot \Big(- a_2\rho - \frac{1}{2} b_2 \overline\rho + r_2\Big)\,dx\Big|\\
\le \mathcal O(1) \Big\|a_1\rho + \frac{1}{2} b_1 \overline\rho + r_1\Big\|_{L^2(\Omega)} \Big\|- a_2\rho - \frac{1}{2} b_2 \overline\rho + r_2\Big\|_{L^2(\Omega)}\quad\text{as}\quad \tau\to\infty.
\end{multline*}
Then, by \eqref{eqn 4.2}, \eqref{eqn 4.4} and \eqref{eqn 4.6},
$$
\Big|\int_\Omega (\mu_1 - \mu_2) e^{i\xi\cdot x}\Big(a_1\rho + \frac{1}{2} b_1 \overline\rho + r_1\Big)\cdot \Big(- a_2\rho - \frac{1}{2} b_2 \overline\rho + r_2\Big)\,dx\Big| = \mathcal O(1)\quad\text{as $\tau\to\infty$}.
$$
In a similar way, we obtain
\begin{align*}
\Big|\int_\Omega \frac{(\sigma_1 - \sigma_2)}{\sigma_1\sigma_2} e^{i\xi\cdot x}\Big(\nabla_{\zeta^1_1}a_1\times \rho+\frac{1}{2}\nabla_{\zeta^1_1}b_1\times \overline\rho\Big)\cdot \Big(\nabla_{-\zeta^2_1}a_2\times \rho+\frac{1}{2}\nabla_{-\zeta^2_1}b_2\times \overline\rho\Big)\,dx\Big| &= \mathcal O(1),\\
\Big|\int_\Omega \frac{(\sigma_1 - \sigma_2)}{\sigma_1\sigma_2} e^{i\xi\cdot x}\Big(\nabla_{\zeta^1_1}a_1\times \rho+\frac{1}{2}\nabla_{\zeta^1_1}b_1\times \overline\rho\Big)\cdot (b_2 \rho_1\times\rho_2)\,dx\Big| &= \mathcal O(1),\\
\Big|\int_\Omega \frac{(\sigma_1 - \sigma_2)}{\sigma_1\sigma_2} e^{i\xi\cdot x} (b_1 \rho_1\times \rho_2) \cdot \Big(\nabla_{-\zeta^2_1}a_2\times \rho+\frac{1}{2}\nabla_{-\zeta^2_1}b_2\times \overline\rho\Big)\,dx\Big| &= \mathcal O(1),\\
\Big|\int_\Omega \frac{(\sigma_1 - \sigma_2)}{\sigma_1\sigma_2} e^{i\xi\cdot x}\Big(\nabla_{\zeta^1_1}a_1\times \rho+\frac{1}{2}\nabla_{\zeta^1_1}b_1\times \overline\rho\Big)\cdot (\nabla_{-\zeta^2}\times r_2)\,dx\Big| &= o(\tau),\\
\Big|\int_\Omega \frac{(\sigma_1 - \sigma_2)}{\sigma_1\sigma_2} e^{i\xi\cdot x} (\nabla_{\zeta^1}\times r_1) \cdot \Big(\nabla_{-\zeta^2_1}a_2\times \rho+\frac{1}{2}\nabla_{-\zeta^2_1}b_2\times \overline\rho\Big)\,dx\Big| &= o(\tau),\\
\Big|\int_\Omega \frac{(\sigma_1 - \sigma_2)}{\sigma_1\sigma_2} e^{i\xi\cdot x} (b_1 \rho_1\times \rho_2) \cdot (\nabla_{-\zeta^2}\times r_2)\,dx\Big| &= o(\tau),\\
\Big|\int_\Omega \frac{(\sigma_1 - \sigma_2)}{\sigma_1\sigma_2} e^{i\xi\cdot x} (\nabla_{\zeta^1}\times r_1) \cdot (b_2 \rho_1\times\rho_2)\,dx\Big| &= o(\tau)
\end{align*}
as $\tau\to\infty$. Here, we used again that $\zeta^j_1 = \mathcal O(1)$ as $\tau\to\infty$, $j=1,2$. Finally, we use \eqref{eqn 4.2}-\eqref{eqn 4.5}, to show that
\begin{align*}
\Big|\int_\Omega \frac{(\sigma_1 - \sigma_2)}{\sigma_1\sigma_2} & e^{i\xi\cdot x} b_1 b_2 \, dx - \int_\Omega \frac{(\sigma_1 - \sigma_2)}{\sigma_1\sigma_2} e^{i\xi\cdot x} \sigma_1^{1/2}\sigma_2^{1/2} e^{\Psi_1+\Psi_2}\,dx\Big| \\
&\le \mathcal O(1) \big\| e^{\alpha_1^\sharp/2 + \alpha_2^\sharp/2} e^{\Psi^\sharp_1+\Psi^\sharp_2} - \sigma_1^{1/2}\sigma_2^{1/2} e^{\Psi_1+\Psi_2} \big\|_{L^2(\Omega)}\\
&\le \mathcal O(1) \big\| e^{\alpha_1^\sharp/2 + \alpha_2^\sharp/2} e^{\Psi^\sharp_1+\Psi^\sharp_2} - \sigma_1^{1/2}\sigma_2^{1/2} e^{\Psi^\sharp_1+\Psi^\sharp_2} \big\|_{L^2(\Omega)}
+ \mathcal O(1) \big\| \sigma_1^{1/2}\sigma_2^{1/2} e^{\Psi^\sharp_1+\Psi^\sharp_2} - \sigma_1^{1/2}\sigma_2^{1/2} e^{\Psi_1+\Psi_2} \big\|_{L^2(\Omega)}\\
&\le \mathcal O(1) \big\| e^{\alpha_1^\sharp/2 + \alpha_2^\sharp/2} - e^{(\log \sigma_1)/2 +(\log\sigma_2)/2} \big\|_{L^2(\Omega)}
+ \mathcal O(1) \big\| e^{\Psi^\sharp_1+\Psi^\sharp_2} - e^{\Psi_1+\Psi_2} \big\|_{L^2(\Omega)}\\
&\le \mathcal O(1) \big\| \alpha_1^\sharp- \log \sigma_1\|_{L^2(\Omega)} + \mathcal O(1) \|\alpha_2^\sharp  - \log\sigma_2 \big\|_{L^2(\Omega)}
 + \mathcal O(1) \big\| \Psi^\sharp_1 - \Psi_1 \|_{L^2(\Omega)} + \mathcal O(1) \big\| \Psi^\sharp_2 - \Psi_2 \big\|_{L^2(\Omega)}\\[0.15cm]
&= o(1)\quad\text{as}\quad \tau\to\infty,
\end{align*}
employing the basic inequality, 
\begin{equation}\label{technical ineq from complex analysis}
|e^z - e^w| \le |z - w| e^{\max(\Re z,\Re w)},\quad z,w\in \C
\end{equation}
from \cite{krupchyk2014uniqueness}. According to these estimates, taking the limit as $\tau \to \infty$ in \eqref{integral identity to find mu}, we come to
$$
\int_{\R^3} e^{i\xi\cdot x} \frac{(\sigma_1 - \sigma_2)}{\sigma_1^{1/2} \sigma_2^{1/2}} e^{\Psi_1 + \Psi_2}\,dx = 0.
$$
Note that the integration is extended to all of $\R^3$ since $\sigma_1 - \sigma_2 = 0$ on $\R^3\setminus\overline\Omega$. This implies that $\sigma_1=\sigma_2$.\smallskip

Next, we set $\sigma = \sigma_1=\sigma_2$. By Proposition~\ref{prop CGO}, there are complex geometric optics solutions $H_1,H_2\in H^1_{\Div}(\Omega)$, with $\nabla\times H_1, \nabla\times H_2 \in H^1_{\Div}(\Omega)$ satisfying
$$
\nabla\times(\sigma^{-1}\nabla\times H_1) - i\omega \mu_1 H_1 = 0\quad\text{and}\quad \nabla\times(\sigma^{-1}\nabla\times H_2) - i\omega \mu_2 H_2 = 0 \quad\text{in}\quad\Omega,
$$
respectively, which have the following forms
$$
H_1(x;\zeta^1) = e^{i\zeta^1\cdot x} \Big(a_1 \rho + r_1\Big),\quad H_2(x;\zeta^2) = e^{-i\zeta^2\cdot x} \Big(-a_2 \rho - \frac{1}{2} b_2 \overline\rho + r_2\Big),
$$
where
$$
a_1 = e^{-\alpha^\sharp(x;\tau)/2},\quad a_2 = e^{-\alpha^\sharp(x;\tau)/2},\quad b_2 = e^{\alpha^\sharp(x;\tau)/2} e^{\Psi^\sharp(x, \rho; \tau)}.
$$
The function $\Psi^\sharp(\cdot, \rho; \tau)\in C^\infty(\R^3)$ satisfies
\begin{equation}\label{eqn 4.7}
\|\Psi^\sharp\|_{W^{1,\infty}(\R^3)} = \mathcal O(1) \quad \text{and}\quad  \|\Psi^\sharp - \Psi\|_{L^2_{\rm loc}(\R^3)} = o(1) \quad\text{as}\quad\tau\to\infty,
\end{equation}
where
$$
\Psi(\cdot, \rho) := - N_\rho^{-1} \{\overline{\rho} \cdot \nabla\log(\sigma\mu_2)^{1/2}\} \in L^\infty(\R^3),\quad j=1,2.
$$
Furthermore, the function $\alpha^\sharp(\,\cdot\,;\tau) \in C^\infty(\R^3)$ satisfies
\begin{equation}\label{eqn 4.8}
\| \alpha^\sharp \|_{W^{2,\infty}(\R^3)} = \mathcal O(1) \quad\text{and}\quad  \| \alpha^\sharp - \log\sigma \|_{W^{2,\infty}(\R^3)} = o(1) \quad\text{as}\quad\tau\to\infty.
\end{equation}
The correction terms $r_1$, $r_2 \in H_{\Div}^1(\Omega)$ satisfy
\begin{equation}\label{eqn 4.9}
\|r_j\|_{L^2(\Omega)} = o(1)\quad\text{and}\quad\|\nabla_{\zeta^j}\times r_j\|_{L^2(\Omega)} = o(\tau) \quad \text{as}\quad \tau\to\infty,\quad j=1,2.
\end{equation}
Substituting $H_1$, $H_2$ and $\sigma = \sigma_1=\sigma_2$ into \eqref{main integral identity}, we come to
\begin{align*}
-\int_\Omega (\mu_1 - \mu_2) e^{i\xi\cdot x} a_1 b_2\,dx &+ \int_\Omega (\mu_1 - \mu_2) e^{i\xi\cdot x} a_1 \rho\cdot r_2\,dx\\
-&\int_\Omega (\mu_1 - \mu_2) e^{i\xi\cdot x} r_1\cdot \Big(a_2\rho+\frac{1}{2}b_2\overline\rho\Big)\,dx + \int_\Omega (\mu_1 - \mu_2) e^{i\xi\cdot x} r_1\cdot r_2\,dx= 0.
\end{align*}
By the Cauchy-Schwartz inequality together with \eqref{eqn 4.8} and \eqref{eqn 4.9},
$$
\Big|\int_\Omega (\mu_1 - \mu_2) e^{i\xi\cdot x} a_1 \rho\cdot r_2\,dx\Big| \le \mathcal O(1) \int_\Omega | a_1 \rho\cdot r_2 |\,dx \le \mathcal O(1) \|a_1\|_{L^2(\Omega)} \|r_2\|_{L^2(\Omega)} = o(1)
$$
as $\tau\to\infty$. In a similar way, and also using \eqref{eqn 4.7}, one can show that
$$
\Big|\int_\Omega (\mu_1 - \mu_2) e^{i\xi\cdot x} r_1\cdot \Big(a_2\rho+\frac{1}{2}b_2\overline\rho\Big)\,dx\Big| = o(1)\quad\text{and}\quad \Big| \int_\Omega (\mu_1 - \mu_2) e^{i\xi\cdot x} r_1\cdot r_2\,dx \Big| = o(1)\quad\text{as}\quad \tau \to\infty.
$$
Finally, using \eqref{eqn 4.7} and \eqref{eqn 4.8},
$$
\Big| \int_\Omega (\mu_1 - \mu_2) e^{i\xi\cdot x} a_1 b_2\,dx - \int_{\R^3} (\mu_1 - \mu_2) e^{i\xi\cdot x} e^{\Psi}\,dx \Big| \le \mathcal O(1) \| e^{\Psi^\sharp} - e^{\Psi}\|_{L^2(\Omega)} \le \mathcal O(1) \|\Psi^\sharp - \Psi\|_{L^2(\Omega)} = o(1)\quad\text{as}\quad \tau \to\infty.
$$
Here, we have again employed inequality \eqref{technical ineq from complex analysis}. Thus, letting $\tau\to\infty$, we obtain
$$
\int_{\R^3} e^{i\xi\cdot x} (\mu_1 - \mu_2) e^{\Psi}\,dx = 0.
$$
The integration is extended to all of $\R^3$ since $\mu_1 - \mu_2 = 0$ on $\R^3\setminus\overline\Omega$. This implies that $\mu_1 = \mu_2$ completing the proof of Theorem~\ref{main thm}.

\section{Reflection approach} \label{section::reflection approach}

In this section, we use Isakov's reflection approach \cite{isakov2007uniqueness} to prove the following local uniqueness result where the region of the boundary that is inaccessible for measurements is a part of a plane. For a closed $\Gamma\subset\p\Omega$, define
$$
C_{\Gamma}(\sigma, \mu; \omega) := \{(\bt(H) |_{\Gamma}, \bt(E)|_{\Gamma}): (H, E)\in H^1_{\Div}(\Omega)\times H^1_{\Div}(\Omega)\text{ is a solution to \eqref{eqn::Maxwell} with }\supp(\bt(H))\subseteq\Gamma\}.
$$

\begin{Theorem}\label{main thm flat}
Let $\Omega\subset \{x\in \R^3: x_3 < 0 \}$ be a bounded domain with $C^{1,1}$ boundary and let $\Gamma_0 = \p\Omega\cap \{x\in \R^3: x_3 = 0\}$ and $\Gamma = \overline{\p\Omega\setminus \Gamma_0}$. Suppose that $\sigma_j,\mu_j\in C^2(\overline\Omega)$, $j=1,2$, satisfy $\sigma_j \ge \sigma_0$ and $\mu_j \ge \mu_0$, for some constants $\sigma_0, \mu_0 > 0$, and
\begin{equation}\label{boundary assumption on Gamma flat}
\p^\alpha \sigma_1|_{\Gamma} = \p^\alpha \sigma_2|_{\Gamma}\quad\text{and}\quad \p^\alpha \mu_1|_{\Gamma} = \p^\alpha \mu_2|_{\Gamma}\quad\text{for}\quad|\alpha| \le 2.
\end{equation}
In addition, assume that $\sigma_j$ and $\mu_j$, $j=1,2$, can be extended into $\R^3$ as $C^2$ functions which are invariant under reflection across the plane $\{x\in \R^3: x_3 = 0\}$. Then $C_\Gamma(\sigma_1, \mu_1; \omega) = C_\Gamma(\sigma_2, \mu_2; \omega)$ implies $\sigma_1=\sigma_2$ and $\mu_1=\mu_2$.
\end{Theorem}

\noindent
Similar results were obtained for the inverse conductivity problem in
\cite{isakov2007uniqueness} and for the IEMP in
\cite{caro2009inverse}. Consider the reflected domain
$$
\Omega^* := \{(x_1, x_2, -x_3)\in\R^3 : (x_1, x_2, x_3)\in\Omega\}
$$
and define
$$
\mathcal U := \Omega \cap \Gamma_0^{\rm int} \cap \Omega^*.
$$
By the assumptions in Theorem~\ref{main thm flat}, we can extend the coefficients $\sigma_j$ and $\mu_j$ into $\mathcal U$ as $C^2$ functions which are even with respect to $x_3$ for $j = 1,2$.  Next, by the assumption \eqref{boundary assumption on Gamma flat}, we can extend $\sigma_j$ and $\mu_j$, $j=1,2$, to $C^2$ functions defined on $\R^3$, still denoted by $\sigma_j$ and $\mu_j$, such that $\sigma_j \ge \sigma_0$, $\mu_j \ge \mu_0$ on $\R^3$, $\sigma_j - \sigma_0, \mu_j - \mu_0 \in C^2_0(\R^3)$ and $\sigma_1=\sigma_2$ and $\mu_1=\mu_2$ on $\R^3\setminus\overline{\mathcal U}$.

\begin{Proposition}\label{prop main integral identity local}
Let $\Omega\subset \{x\in \R^3: x_3 < 0 \}$ be a bounded domain with $C^{1,1}$ boundary . Let $\Gamma_0 := \p\Omega\cap \{x\in\R^3: x_3=0\}$ and $\Gamma:=\overline{\p\Omega\setminus \Gamma_0}$. Suppose that
\begin{equation}\label{local data assumption}
Z_{\sigma_1,\mu_1}^\omega(f)|_{\Gamma} = Z_{\sigma_2,\mu_2}^\omega(f)|_{\Gamma}\quad\text{for all}\quad f\in TH^{1/2}_{\Div}(\p\Omega)\quad\text{with}\quad \supp(f)\subset\Gamma;
\end{equation}
then
\begin{equation}\label{main integral identity local}
\int_\Omega (\mu_1 - \mu_2) H_1\cdot H_2\,dx + \frac{1}{i\omega}\int_\Omega \frac{(\sigma_1 - \sigma_2)}{\sigma_1\sigma_2} (\nabla\times H_1)\cdot (\nabla\times H_2)\,dx = 0
\end{equation}
for all $H_j\in H_{\Div}^1(\Omega)$ with $\nabla\times H_j\in H_{\Div}^1(\Omega)$ solving
$$
\nabla\times(\sigma_j^{-1}\nabla\times H_j) - i\omega \mu_j H_j = 0\quad\text{in}\quad\Omega,\quad j=1,2.
$$
and satisfying $\supp(\bt(H_j))\subseteq \Gamma$.
\end{Proposition}
\begin{proof}
Similarly as in the proof of  \eqref{main integral identity}, define
$$
E_j := \sigma_j^{-1} \nabla\times H_j,\quad j=1,2.
$$
Then $E_j\in H_{\Div}^1(\Omega)$ and $\nabla\times E_j = i\omega\mu_j H_j$. Hence $(H_j, E_j) \in H^1_{\Div}(\Omega) \times H^1_{\Div}(\Omega)$, $j=1,2$, solve
$$
\nabla\times E_j=i\omega\mu_j H_j\quad\text{and}\quad \nabla\times H_j=\sigma_j E_j \quad\text{in}\quad\Omega,\quad j=1,2.
$$
Then by the assumption \eqref{local data assumption}, there is $(H', E')\in H^1_{\Div}(\Omega) \times H^1_{\Div}(\Omega)$ with $\supp(\bt(H'))\subseteq\Gamma$ satisfying
$$
\nabla\times E'=i\omega\mu_2 H'\quad\text{and}\quad \nabla\times H'=\sigma_2 E'\quad\text{in}\quad\Omega
$$
and
$$
\bt(H')|_{\Gamma} = \bt(H_1)|_{\Gamma}\quad\text{and}\quad\bt(E')|_{\Gamma} = \bt(E_1)|_{\Gamma}.
$$
Integrating by parts, leads to
\begin{align*}
\int_\Omega \nabla\times(H'-H_1)\cdot E_2\,dx - \int_\Omega i\omega\mu_2 (H'-H_1)\cdot H_2\,dx&= \int_\Omega \nabla\times(H'-H_1)\cdot E_2\,dx - \int_\Omega (H'-H_1)\cdot \nabla\times E_2\,dx\\
& = \int_{\p\Omega} \bt (H'-H_1)\cdot E_2\,dS(x) = \int_{\Gamma_0^{\rm int}} \bt (H'-H_1)\cdot \bt(E_2)\,dS(x) = 0,
\end{align*}
since both $\bt(H')$ and $\bt(H_1)$ are supported on $\Gamma$. Similarly,
$$
\int_\Omega \nabla\times(E'-E_1)\cdot H_2\,dx - \int_\Omega \sigma_2 (E'-E_1)\cdot E_2\,dx = \int_{\Gamma_0^{\rm int}} \bt (E'-E_1)\cdot \bt(H_2)\,dS(x) = 0,
$$
since $\bt(H_2)$ is supported on $\Gamma$. The remainder of the proof of \eqref{main integral identity local} is similar to that of \eqref{main integral identity}.
\end{proof}

For $\beta:\R^3\to \C$ and $X:\R^3\to \C^3$, we define the reflections as
$$
\beta^*(x) := \beta(x_1, x_2, - x_3),\quad X^*(x) := \big(X_1(x_1, x_2, - x_3), X_2(x_1, x_2, - x_3), - X_3(x_1, x_2, - x_3)\big)
$$
with the properties,
\begin{equation}\label{eqn::reflection identities}
\nabla \beta^* = (\nabla \beta)^*,\quad (\beta X)^* = \beta^* X^*,\quad \nabla\times X^* = - (\nabla\times X)^*.
\end{equation}

\subsection* {Proof of Uniqueness}

Consider $\zeta^1$ and $\zeta^2$ defined as in the proof of Theorem~\ref{main thm}.
Then, by Proposition~\ref{prop CGO}, there are complex geometric optics solutions $\tilde H_1, \tilde H_2\in H^1_{\Div}(\mathcal U)$, with $\nabla\times \tilde H_1, \nabla\times \tilde H_2 \in H^1_{\Div}(\mathcal U)$, for
$$
\nabla\times(\sigma_1^{-1}\nabla\times\tilde H_1) - i\omega \mu_1\tilde H_1 = 0\quad\text{and}\quad \nabla\times(\sigma_2^{-1}\nabla\times\tilde H_2) - i\omega \mu_2\tilde H_2 = 0 \quad\text{in}\quad\mathcal U,
$$
respectively, which have the following forms
$$
\tilde H_1(x;\zeta^1) = e^{i\zeta^1\cdot x} \Big(a_1 \rho + \frac{1}{2} b_1 \overline\rho + r_1\Big),\quad \tilde H_2(x;\zeta^2) = e^{-i\zeta^2\cdot x} \Big(-a_2 \rho - \frac{1}{2} b_2 \overline\rho + r_2\Big),
$$
where
\begin{align*}
a_1 &= e^{-\alpha_1^\sharp(x;\tau)/2},\quad b_1 = e^{\alpha_1^\sharp(x;\tau)/2} e^{\Psi^\sharp_1(x, \rho; \tau)},\\
a_2 &= e^{-\alpha_2^\sharp(x;\tau)/2},\quad b_2 = e^{\alpha_2^\sharp(x;\tau)/2} e^{\Psi^\sharp_2(x, \rho; \tau)}.
\end{align*}
The functions $\Psi_1^\sharp(\cdot, \rho; \tau),\Psi_2^\sharp(\cdot, \rho; \tau)\in C^\infty(\R^3)$, $\alpha_1^\sharp(\,\cdot\,;\tau), \alpha_2^\sharp(\,\cdot\,;\tau) \in C^\infty(\R^3)$ and the correction terms $r_1$, $r_2 \in H^1_{\Div}(\mathcal U)$ satisfy \eqref{eqn 4.3}-\eqref{eqn 4.6}. Using \eqref{eqn::reflection identities}, it follows that
$$
\nabla\times (\sigma_j^{-1}\nabla\times \tilde H_j^*) - i\omega \mu_j \tilde H_j^* = - \nabla\times (\sigma_j^{-1}\nabla\times\tilde H_j)^* - (i\omega \mu_j \tilde H_j)^* = \big(\nabla\times (\sigma_j^{-1}\nabla\times\tilde H_j) - i\omega \mu_j\tilde H_j\big)^* = 0\quad\text{in}\quad \mathcal U.
$$
Therefore,
$$
H_1(x;\rho) := \tilde H_1(x;\rho) - \tilde H_1^*(x;\rho),\quad H_2(x;\rho) := \tilde H_2(x;\rho) - \tilde H_2^*(x;\rho).
$$
also satisfy
$$
\nabla\times(\sigma_1^{-1}\nabla\times\tilde H_1) - i\omega \mu_1\tilde H_1 = 0\quad\text{and}\quad \nabla\times(\sigma_2^{-1}\nabla\times\tilde H_2) - i\omega \mu_2\tilde H_2 = 0 \quad\text{in}\quad\mathcal U.
$$
As in the proof of Proposition~\ref{prop CGO}, it is not difficult to see that $H_1|_\Omega$ and $H_2|_\Omega$, still denoted by $H_1$ and $H_2$, respectively, belong to $H^1_{\Div}(\Omega)$. Moreover, $\nabla\times H_1, \nabla\times H_2 \in H^1_{\Div}(\Omega)$. Using the fact that $\nu = (0, 0, 1)$ on $\Gamma$, we have  $\nu\times H_1|_{\Gamma}=\nu\times H_2|_{\Gamma}=0$. Thus, $H_1$ and $H_2$ satisfy the hypotheses of Proposition~\ref{prop main integral identity local}, and hence we can substitute $H_1$ and $H_2$ into \eqref{main integral identity local}. To that end, we first compute $\nabla\times H_1$ and $\nabla\times H_2$. Using \eqref{eqn::reflection identities},
\begin{align*}
\nabla\times H_1 &= \nabla\times \tilde H_1 + (\nabla\times \tilde H_1)^* \\
&= e^{i\zeta^1\cdot x} \Big(\nabla_{\zeta^1_1}a_1\times \rho+\frac{1}{2}\nabla_{\zeta^1_1}b_1\times \overline\rho + b_1 \tau \rho_1\times \rho_2 + \nabla_{\zeta^1}\times r_1\Big)\\
&\quad + e^{i{\zeta^1}^*\cdot x} \Big((\nabla_{\zeta^1_1}a_1\times \rho)^*+\frac{1}{2}(\nabla_{\zeta^1_1}b_1\times \overline\rho)^* + b_1^* \tau (\rho_1\times \rho_2)^* + (\nabla_{\zeta^1}\times r_1)^*\Big).
\end{align*}
Similarly,
\begin{align*}
\nabla\times H_2 &= e^{-i\zeta^2\cdot x} \Big(- \nabla_{-\zeta^2_1}a_2\times \rho - \frac{1}{2}\nabla_{-\zeta^2_1}b_2\times \overline\rho + b_2 \tau \rho_1\times \rho_2 + \nabla_{-\zeta^2}\times r_2\Big)\\
&\quad + e^{-i{\zeta^2}^*\cdot x} \Big(- (\nabla_{-\zeta^2_1}a_2\times \rho)^* - \frac{1}{2}(\nabla_{-\zeta^2_1}b_2\times \overline\rho)^* + b_2^* \tau (\rho_1\times \rho_2)^* + (\nabla_{-\zeta^2}\times r_2)^*\Big).
\end{align*}
We take a closer look at the phases of products of these vector fields. We have
$$
i(\zeta^1 - \zeta^2)\cdot x = i\xi\cdot x,\quad i({\zeta^1}^* - {\zeta^2}^*)\cdot x = i\xi^*\cdot x,
$$
$$
i(\zeta^1 - {\zeta^2}^*)\cdot x = i\tilde\xi_{+}\cdot x - 2\tau\rho_{2,3} x_3 - \eta_{+},\quad i({\zeta^1}^* - \zeta^2)\cdot x = i\tilde\xi_{-}\cdot x + 2\tau\rho_{2,3} x_3 - \eta_{-},
$$
where
$$
\tilde\xi_{\pm} = \Bigg(\xi', \pm 2\tau\sqrt{1 - \frac{|\xi|^2}{4\tau^2}}\rho_{1,3}\Bigg),\quad |\tilde\xi_{\pm}|\to\infty\quad\text{as}\quad \tau\to \infty
$$
and
$$
\eta_{\pm}:= \pm\frac{2\omega \sigma_{0} \mu_{0} \rho_{1,3} x_3}{\tau\Big(\sqrt{1 - \frac{|\xi|^2}{4\tau^2} + \frac{i\omega\sigma_0\mu_0}{\tau^2}} + \sqrt{1 - \frac{|\xi|^2}{4\tau^2}}\Big)} ,\quad |\eta_{\pm}| = \mathcal O(\tau^{-1})\quad\text{as}\quad \tau\to \infty.
$$
Furthermore, we assume that $\rho_{1,3} \neq 0$ and $\rho_{2,3} = 0$. Then
\begin{align*}
\int_\Omega & e^{i\tilde\xi_{+}\cdot x - \eta_{+}} \frac{(\sigma_1 - \sigma_2)}{\sigma_1\sigma_2} \Big(\nabla_{\zeta^1_1} a_1\times \rho + \frac{1}{2}\nabla_{\zeta^1_1}b_1\times \overline\rho\Big) \cdot \Big((\nabla_{-\zeta^2_1}a_2\times \rho)^* + \frac{1}{2}(\nabla_{-\zeta^2_1}b_2\times \overline\rho)^*\Big)\,dx\\
& = \int_\Omega e^{i\tilde\xi_{+}\cdot x} \frac{(\sigma_1 - \sigma_2)}{\sigma_1\sigma_2} \Big\{e^{- \eta_{+}}\Big(\nabla_{\zeta^1_1} a_1\times \rho + \frac{1}{2}\nabla_{\zeta^1_1}b_1\times \overline\rho\Big) \cdot \Big((\nabla_{-\zeta^2_1}a_2\times \rho)^* + \frac{1}{2}(\nabla_{-\zeta^2_1}b_2\times \overline\rho)^*\Big)\\
&\qquad\qquad\qquad\qquad - \Big(\nabla_{\zeta^1_1} \mu_1^{1/2}\times \rho + \frac{1}{2}\nabla_{\zeta^1_1}(\mu^{-1/2}e^{\Psi_1})\times \overline\rho\Big) \cdot \Big((\nabla_{-\zeta^2_1}\mu_2^{1/2}\times \rho)^* + \frac{1}{2}(\nabla_{-\zeta^2_1}(\mu_2^{-1/2}e^{\Psi_2})\times \overline\rho)^*\Big)\Big\}\,dx\\
&\quad + \int_\Omega e^{i\tilde\xi_{+}\cdot x} \frac{(\sigma_1 - \sigma_2)}{\sigma_1\sigma_2} \Big(\nabla_{\zeta^1_1} \mu_1^{1/2}\times \rho + \frac{1}{2}\nabla_{\zeta^1_1}(\mu^{-1/2}e^{\Psi_1})\times \overline\rho\Big) \cdot \Big((\nabla_{-\zeta^2_1}\mu_2^{1/2}\times \rho)^* + \frac{1}{2}(\nabla_{-\zeta^2_1}(\mu_2^{-1/2}e^{\Psi_2})\times \overline\rho)^*\Big)\,dx
\end{align*}
The first integral on the right-side goes to zero as $\tau\to\infty$ according to estimates \eqref{eqn 4.2}-\eqref{eqn 4.5} and the fact that $|\eta_{\pm}| = \mathcal O(\tau^{-1})$ as $\tau\to \infty$. The second integral goes to zero as $\tau\to\infty$ by the Riemann-Lebesgue lemma. Therefore,
$$
\Big|\int_\Omega  e^{i\tilde\xi_{+}\cdot x - \eta_{+}} \frac{(\sigma_1 - \sigma_2)}{\sigma_1\sigma_2} \Big(\nabla_{\zeta^1_1} a_1\times \rho + \frac{1}{2}\nabla_{\zeta^1_1}b_1\times \overline\rho\Big) \cdot \Big((\nabla_{-\zeta^2_1}a_2\times \rho)^* + \frac{1}{2}(\nabla_{-\zeta^2_1}b_2\times \overline\rho)^*\Big)\,dx\Big| = o(1)\quad\text{as}\quad\tau\to\infty.
$$
In a similar way, one can show that
\begin{align*}
\Big|\int_\Omega e^{i\tilde\xi_{-}\cdot x - \eta_{-}} \frac{(\sigma_1 - \sigma_2)}{\sigma_1\sigma_2} \Big((\nabla_{\zeta^1_1}a_1\times \rho)^*+\frac{1}{2}(\nabla_{\zeta^1_1}b_1\times \overline\rho)^*\Big) \cdot \Big(\nabla_{-\zeta^2_1}a_2\times \rho + \frac{1}{2}\nabla_{-\zeta^2_1}b_2\times \overline\rho\Big)\,dx\Big| &= o(1)\\
\Big|\int_\Omega  e^{i\tilde\xi_{+}\cdot x - \eta_{+}} \frac{(\sigma_1 - \sigma_2)}{\sigma_1\sigma_2} \Big(\nabla_{\zeta^1_1} a_1\times \rho + \frac{1}{2}\nabla_{\zeta^1_1}b_1\times \overline\rho\Big) \cdot \Big(b_2^* \tau (\rho_1\times \rho_2)^*\Big)\,dx\Big| &= o(\tau)\\
\Big|\int_\Omega e^{i\tilde\xi_{-}\cdot x - \eta_{-}} \frac{(\sigma_1 - \sigma_2)}{\sigma_1\sigma_2} \Big(b_1^* \tau (\rho_1\times \rho_2)^*\Big) \cdot \Big(\nabla_{-\zeta^2_1}a_2\times \rho + \frac{1}{2}\nabla_{-\zeta^2_1}b_2\times \overline\rho\Big)\,dx\Big| &= o(\tau)\\
\Big|\int_\Omega  e^{i\tilde\xi_{+}\cdot x - \eta_{+}} \frac{(\sigma_1 - \sigma_2)}{\sigma_1\sigma_2} \Big(b_1 \tau \rho_1\times \rho_2\Big) \cdot \Big((\nabla_{-\zeta^2_1}a_2\times \rho)^* + \frac{1}{2}(\nabla_{-\zeta^2_1}b_2\times \overline\rho)^*\Big)\,dx\Big| &= o(\tau)\\
\Big|\int_\Omega e^{i\tilde\xi_{-}\cdot x - \eta_{-}} \frac{(\sigma_1 - \sigma_2)}{\sigma_1\sigma_2} \Big((\nabla_{\zeta^1_1}a_1\times \rho)^*+\frac{1}{2}(\nabla_{\zeta^1_1}b_1\times \overline\rho)^*\Big) \cdot \Big(b_2\tau \rho_1\times\rho_2\Big)\,dx\Big| &= o(\tau)\\
\Big|\int_\Omega  e^{i\tilde\xi_{+}\cdot x - \eta_{+}} \frac{(\sigma_1 - \sigma_2)}{\sigma_1\sigma_2} \Big(b_1 \tau \rho_1\times \rho_2\Big) \cdot \Big(b_2^*\tau (\rho_1\times\rho_2)^*\Big)\,dx\Big| &= o(\tau^2)\\
\Big|\int_\Omega  e^{i\tilde\xi_{-}\cdot x - \eta_{-}} \frac{(\sigma_1 - \sigma_2)}{\sigma_1\sigma_2} \Big(b_1^* \tau (\rho_1\times \rho_2)^*\Big) \cdot \Big(b_2\tau \rho_1\times\rho_2\Big)\,dx\Big| &= o(\tau^2)\\
\Big|\int_\Omega(\mu_1 - \mu_2) e^{i\tilde\xi_{+}\cdot x - \eta_{+}} \Big(a_1 \rho + \frac{1}{2} b_1 \overline\rho\Big)\cdot \Big(a_2^* \rho^* + \frac{1}{2} b_2^* \overline\rho^*\Big)\,dx\Big| &= o(1)\\
\Big|\int_\Omega(\mu_1 - \mu_2) e^{i\tilde\xi_{-}\cdot x - \eta_{-}} \Big(a_1^* \rho^* + \frac{1}{2} b_1^* \overline\rho^*\Big)\cdot \Big(a_2 \rho + \frac{1}{2} b_2 \overline\rho\Big)\,dx\Big| &= o(1)
\end{align*}
as $\tau\to\infty$. Now we substitute $H_1$, $H_2$, $\nabla\times H_1$ and $\nabla\times H_2$ into \eqref{main integral identity local} and divide the whole identity by $\tau^2$. According to the estimates obtained above, the terms with phases $i\tilde\xi_{\pm}\cdot x - \eta_{\pm}$ that do not involve the correction terms $r_1$ and $r_2$ go to zero as $\tau\to\infty$. The terms with phases $i\tilde\xi_{\pm}\cdot x - \eta_{\pm}$ that involve the correction terms $r_1$ and $r_2$ go to zero as $\tau\to\infty$, because of the correction terms. Finally, the terms with phases $i\xi\cdot x$ or $i\xi^*\cdot x$ can be controlled exactly as in the proof of Theorem~\ref{main thm} using \eqref{eqn 4.2} - \eqref{eqn 4.6} and \eqref{technical ineq from complex analysis}. Thus, letting $\tau\to\infty$, we obtain
$$
\int_{\Omega} e^{i\xi\cdot x} \frac{(\sigma_1 - \sigma_2)}{\sigma_1^{1/2} \sigma_2^{1/2}}e^{\Psi_1 + \Psi_2}\,dx + \int_{\Omega} e^{i\xi^*\cdot x} \Bigg(\frac{(\sigma_1 - \sigma_2)}{\sigma_1^{1/2} \sigma_2^{1/2}}e^{\Psi_1 + \Psi_2}\Bigg)^*\,dx = 0.
$$
Making the change of the variables $(x_1, x_2, x_3)\mapsto (x_1, x_2, - x_3)$ in the second integral, we come to
$$
\int_{\mathcal D} e^{i\xi\cdot x} \frac{(\sigma_1 - \sigma_2)}{\sigma_1^{1/2} \sigma_2^{1/2}} e^{\Psi_1 + \Psi_2}\,dx = 0.
$$
This integral can be extended to all of $\R^3$ since $\sigma_1 - \sigma_2 = 0$ on $\R^3\setminus\overline{\mathcal D}$. Then this will imply that $\sigma_1=\sigma_2$ in $\R^3$.\smallskip

Next, we set $\sigma = \sigma_1 = \sigma_2$. By Proposition~\ref{prop CGO}, there are complex geometric optics solutions $\tilde H_1, \tilde H_2\in H^1_{\Div}(\mathcal U)$, with $\nabla\times \tilde H_1, \nabla\times\tilde H_2 \in H^1_{\Div}(\mathcal U)$, for
$$
\nabla\times(\sigma^{-1}\nabla\times\tilde H_1) - i\omega \mu_1\tilde H_1 = 0\quad\text{and}\quad \nabla\times(\sigma^{-1}\nabla\times\tilde H_2) - i\omega \mu_2\tilde H_2 = 0 \quad\text{in}\quad\mathcal U,
$$
respectively, which have the following forms
$$
\tilde H_1(x;\zeta^1) = e^{i\zeta^1\cdot x} (a_1 \rho + r_1),\quad \tilde H_2(x;\zeta^2) = e^{-i\zeta^2\cdot x} \Big(-a_2 \rho - \frac{1}{2} b_2 \overline\rho + r_2\Big),
$$
where
$$
a_1 = e^{-\alpha^\sharp(x;\tau)/2},\quad a_2 = e^{-\alpha^\sharp(x;\tau)/2},\quad b_2 = e^{\alpha^\sharp(x;\tau)/2} e^{\Psi^\sharp(x, \rho; \tau)}.
$$
The functions $\Psi^\sharp(\cdot, \rho; \tau), \alpha^\sharp(\,\cdot\,;\tau)\in C^\infty(\R^3)$ and the correction terms $r_1$, $r_2 \in H_{\Div}^1(\mathcal U)$ satisfy \eqref{eqn 4.7} - \eqref{eqn 4.9}. In a similar way as before, one can show that
$$
H_1(x;\rho) := \tilde H_1(x;\rho) - \tilde H_1^*(x;\rho),\quad H_2(x;\rho) := \tilde H_2(x;\rho) - \tilde H_2^*(x;\rho).
$$
satisfy
$$
\nabla\times(\sigma_1^{-1}\nabla\times\tilde H_1) - i\omega \mu_1\tilde H_1 = 0\quad\text{and}\quad \nabla\times(\sigma_2^{-1}\nabla\times\tilde H_2) - i\omega \mu_2\tilde H_2 = 0 \quad\text{in}\quad\mathcal U.
$$
Also, the restrictions of $H_1$ and $H_2$ onto $\Omega$, still denoted by $H_1$ and $H_2$, respectively, belong to $H^1_{\Div}(\Omega)$ and satisfy $\nu\times H_1|_{\Gamma}=\nu\times H_2|_{\Gamma}=0$. Thus, $H_1$ and $H_2$ satisfy the hypotheses of Proposition~\ref{prop main integral identity local}. We substitude $H_1$, $H_2$ and $\sigma = \sigma_1=\sigma_2$ into \eqref{main integral identity local}. As before, we assume that $\rho_{1,3} \neq 0$ and $\rho_{2,3} = 0$. Therefor, we obtain
\begin{align*}
0 &= \int_\Omega (\mu_1-\mu_2) e^{i\xi\cdot x}(a_1\rho + r_1)\cdot \Big(-a_2\rho-\frac{1}{2}b_2\overline{\rho}+r_2\Big)\,dx\\
&\quad+\int_\Omega (\mu_1-\mu_2) e^{i\xi^*\cdot x}(a_1^*\rho^*+r_1^*)\cdot \Big(-a_2^*\rho^*-\frac{1}{2}b_2^*\overline{\rho}^*+r_2^*\Big)\,dx\\
&\quad-\int_\Omega (\mu_1-\mu_2) e^{i\tilde\xi_+\cdot x-\eta_+}(a_1\rho+r_1)\cdot \Big(-a_2^*\rho^*-\frac{1}{2}b_2^*\overline{\rho}^*+r_2^*\Big)\,dx\\
&\quad-\int_\Omega (\mu_1-\mu_2) e^{i\tilde\xi_-\cdot x-\eta_-}(a_1^*\rho^*+r_1^*)\cdot \Big(-a_2\rho-\frac{1}{2}b_2\overline{\rho}+r_2\Big)\,dx.
\end{align*}
As before, we use \eqref{eqn 4.7} - \eqref{eqn 4.8}, the fact that $|\eta_{\pm}| = \mathcal O(\tau^{-1})$ as $\tau\to \infty$ and the Riemann-Lebesgue lemma, to show that
\begin{align*}
\Big|\int_\Omega (\mu_1-\mu_2) e^{i\tilde\xi_+\cdot x-\eta_+}a_1\rho\cdot \Big(-a_2^*\rho^*-\frac{1}{2}b_2^*\overline{\rho}^*\Big)\,dx\Big| &= o(1),\\
\Big|\int_\Omega (\mu_1-\mu_2) e^{i\tilde\xi_-\cdot x-\eta_-}a_1^*\rho^*\cdot \Big(-a_2\rho-\frac{1}{2}b_2\overline{\rho}\Big)\,dx\Big| &= o(1)
\end{align*}
as $\tau\to\infty$. These estimates guarantee that the terms with phases $i\tilde\xi_{\pm}\cdot x - \eta_{\pm}$ that do not involve the correction terms $r_1$ and $r_2$ go to zero as $\tau\to\infty$. The terms with phases $i\tilde\xi_{\pm}\cdot x - \eta_{\pm}$ that involve the correction terms $r_1$ and $r_2$ also go to zero as $\tau\to\infty$ by \eqref{eqn 4.9}. Finally, the terms with phases $i\xi\cdot x$ or $i\xi^*\cdot x$ can be controlled exactly as in the proof of Theorem~\ref{main thm} using \eqref{eqn 4.7} - \eqref{eqn 4.9} and \eqref{technical ineq from complex analysis}. Thus, letting $\tau\to\infty$, we obtain
$$
\int_{\Omega} e^{i\xi\cdot x} (\mu_1 - \mu_2) e^{\Psi}\,dx + \int_{\Omega} e^{i\xi^*\cdot x} \Big((\mu_1 - \mu_2) e^{\Psi}\Big)^*\,dx = 0.
$$
Making the change of the variables $(x_1, x_2, x_3)\mapsto (x_1, x_2, - x_3)$ in the second integral, we come to
$$
\int_{\mathcal U} e^{i\xi\cdot x} (\mu_1 - \mu_2) e^{\Psi}\,dx = 0
$$
This integral can be extended to all of $\R^3$ since $\mu_1 - \mu_2 = 0$ on $\R^3\setminus\overline\Omega$. This implies that $\mu_1 = \mu_2$ completing the proof of Theorem~\ref{main thm flat}.

\section{Proof of Theorem~\ref{main thm 2}}\label{section::proof of thm 2}

Without loss of generality, we can assume that $B_0$ is the open ball of radius $1/2$ centered at $x_0 = (0,0,1/2)$ and that $0\notin \overline\Omega$. We recall that
$$
\Gamma_0 = \p\Omega\cap \p B_0,\quad \Gamma_0 \neq \p B_0\quad\text{and}\quad\Gamma = \overline{\p\Omega\setminus \Gamma_0}.
$$
We define the map
$$
K:\Omega\to \R^3\setminus\{0\},\quad K(x) := |x|^{-2}x
$$
which is known as the Kelvin transform. One can easily verify that $K^{-1}(y) = |y|^{-2}y$ for $y\in K(\Omega)$. We let $DK$ and $DK^{-1}$ denote the Jacobi matrices of $K$ and $K^{-1}$, respectively.

\smallskip

Next, we define $\widetilde\Omega := \{- y + x_0 : y\in K(\Omega)\}$ and
$$
 F : \widetilde\Omega\to \Omega,\quad F(y) := - K^{-1}(y) + x_0,\quad y\in\widetilde\Omega.
$$
Then
$$
F^{-1}(x) = - K(x-x_0),\quad x\in \Omega.
$$
It is not difficult to verify that $\widetilde\Omega\subset \{x\in
\R^3: x_3 < 0 \}$ and $\widetilde\Gamma_0:=F^{-1}(\Gamma_0)$ is a
subset of the plane $\{x\in \R^3: x_3 = 0\}$. We also write
$\widetilde\Gamma:=F^{-1}(\Gamma)$. Thus, we are in a situation when
the inaccessible part of the boundary is part of a plane. A direct
calculation gives
$$
DF^{-1} = - |x-x_0|^{-2} I + 2 |x-x_0|^{-4} (x-x_0)(x-x_0)^T ,\quad DF = - |y|^{-2} I + 2 |y|^{-4} yy^T,
$$
where $x-x_0$ and $y$ are considered as column vectors and $I$ is the $3\times 3$ identity matrix. These identities can be used to show that
\begin{equation}\label{eqn::DFDF^T}
DF=(DF)^T\quad\text{and}\quad DF(DF)^T=|y|^{-4}I
\end{equation}
and
\begin{equation}\label{eqn::key identities for the Kelvin transform}
DF^{-1}\circ F = |y|^{4} DF,\quad DF = (DF)^T \quad\text{and}\quad \det(DF) = |y|^{-6}.
\end{equation}

\begin{Lemma}\label{invariance of maxwell equation under pullback}
Let $(H_j, E_j)\in H^1_{\Div}(\Omega)\times H^1_{\Div}(\Omega)$, $j=1,2$. Consider their pullbacks onto $\Omega$,
$$
\widetilde H_j := F^* H_j,\quad \widetilde E_j := F^* E_j,\quad\widetilde\sigma_j = \sigma_j\circ F,\quad \widetilde\mu_j = \mu_j\circ F,\quad j=1,2.
$$
Then
$$
\nabla\times E_j=i\omega\mu_j H_j\quad\text{and}\quad \nabla\times H_j=\sigma_j E_j \quad\text{in}\quad\Omega
$$
if and only if
$$
\widetilde\nabla\times \widetilde E_j=i\omega|y|^{-2}\widetilde\mu_j \widetilde H_j\quad\text{and}\quad \widetilde\nabla\times \widetilde H_j=|y|^{-2}\widetilde\sigma_j\widetilde E_j \quad\text{in}\quad\widetilde\Omega.
$$
\end{Lemma}

\noindent
Here and in what follows, $\widetilde\nabla\times$ denotes the curl operator with respect to the coordinates in $\widetilde\Omega$.
\begin{proof}
The claim of the lemma is easy to prove by straightforward calculations using \eqref{eqn::key identities for the Kelvin transform} and the facts from Appendix~\ref{section::pullbacks}.
\end{proof}

\begin{Lemma}\label{impedance map under pullback}
The following holds true,
\[
   Z_{|y|^{-2}\widetilde\sigma_j ,|y|^{-2}\widetilde\mu_j}^\omega(F^*f)
        = F^*\big(Z_{\sigma_j ,\mu_j}^\omega(f)\big)
\]
for all $f \in TH^{1/2}_{\Div}(\p\Omega)$, $j=1,2$.
\end{Lemma}

\begin{proof}
Consider a $C^{1,1}$ boundary defining function $\rho$ for $\p\Omega$. Then using \eqref{eqn::DFDF^T},
$$
\widetilde\nu = - \frac{\nabla(\rho\circ F)}{|\nabla(\rho\circ F)|}\Bigg|_{\p\widetilde\Omega} = - \frac{(DF)^T (\nabla \rho)\circ F}{|(DF)^T (\nabla \rho)\circ F|}\Bigg|_{\p\widetilde\Omega} = - \frac{|y|^2 F^*(\nabla \rho)}{|(\nabla\rho)\circ F|}\Bigg|_{\p\widetilde\Omega} = |y|^2 F^*\Bigg(-\frac{\nabla \rho}{|\nabla\rho|}\Bigg)\Bigg|_{\p\widetilde\Omega} = |y|^2 F^*\nu.
$$
Using Lemma~\ref{lemma::pullback of the cross product} and \eqref{eqn::key identities for the Kelvin transform},
$$
\tilde\nu\times\tilde H_j = |y|^2 F^*\nu\times F^*H_j = |y|^{-4} \big((DF)^{-1}(\nu\times H_j)\big)\circ F = DF\,\big((\nu\times H_j)\circ F\big) = (DF)^T\big((\nu\times H_j)\circ F\big) = F^*(\nu\times H_j).
$$
Similarly, $\tilde\nu\times\tilde E_j =  F^*(\nu\times E_j)$.
\end{proof}

With $Z_{\sigma_1 ,\mu_1}^\omega = Z_{\sigma_2 ,\mu_2}^\omega$, it follows from these lemmas that
$$
Z_{|y|^{-2}\widetilde\sigma_1 ,|y|^{-2}\widetilde\mu_1}^\omega = Z_{|y|^{-2}\widetilde\sigma_2 ,|y|^{-2}\widetilde\mu_2}^\omega.
$$
Finally, by hypothesis, $\sigma_j$ and $\mu_j$, $j=1,2$, can be extended to $\R^3$ as $C^2$ functions which are invariant under reflection across $\p B_0$. This is equivalent to the invariance of such extensions of $\sigma_j$ and $\mu_j$, $j=1,2$, under the map $x\mapsto F\circ R\circ F^{-1}(x)$, where $R(y_1,y_2,y_3)=(y_1,y_2,-y_3)$. Therefore, $\tilde\sigma_j$ and $\tilde\mu_j$ can be extended into $\R^3$ as $C^2$ functions which are invariant under reflection across the plane $\{x\in\R^3: x_3=0\}$. Thus, by Theorem~\ref{main thm flat}, we get $|y|^{-2}\tilde\sigma_1=|y|^{-2}\tilde\sigma_2$ and $|y|^{-2}\tilde\mu_1=|y|^{-2}\tilde\mu_2$ in $\widetilde\Omega$. Hence, $\sigma_1=\sigma_2$ and $\mu_1=\mu_2$ in $\Omega$ as desired.

\section*{Ackkowledgments}

YA would like to thank Total E \& P Research \& Technology USA, for
financial support. MVdH was supported by the Simons Foundation under
the MATH + X program, the National Science Foundation under grant
DMS-1815143, and the corporate members of the Geo-Mathematical Imaging
Group at Rice University.

\appendix




\section{Identities with pullbacks} \label{section::pullbacks}

Suppose that $\Omega,\widetilde\Omega\subset \R^3$ are bounded domains and $F:\widetilde\Omega\to \Omega$ is a $C^1$ bijective map. For a given $u\in H^1_{\Div}(\Omega)$, the pullback $F^* u\in H^1_{\Div}(\widetilde\Omega)$ is defined as
$$
F^* u := (DF)^T(u\circ F),
$$
where $DF$ is the Jacobi matrix of $F$. According to \cite[Corollary~3.58]{monk2003finite},
$$
\widetilde\nabla\times (F^* u) = \big(\det(DF) (DF)^{-1} \nabla\times u\big)\circ F.
$$
\begin{Lemma}\label{lemma::pullback of the cross product}
Suppose $u,v\in H^1_{\Div}(\Omega)$ and $F:\widetilde\Omega\to \Omega$ is a $C^1$ bijective map as before. Then
$$
F^* u\times F^* v = \det(DF) \big((DF)^{-1}(u\times v)\big)\circ F.
$$
\end{Lemma}
\begin{proof}
Let $x$ and $y$ be coordinate systems in $\Omega$ and $\widetilde\Omega$, respectively. By definition,
$$
u_i = \sum_{k=1}^3 \frac{\p y_k}{\p x_i} (F^* u)_k\circ F^{-1}\quad\text{and}\quad v_j = \sum_{l=1}^3 \frac{\p y_l}{\p x_j} (F^* v)_l\circ F^{-1}.
$$
Then
$$
u_i v_j - v_i u_j = \sum_{k=1}^3 \sum_{l=1}^3  \frac{\p y_k}{\p x_i} \big((F^* u)_k (F^* v)_l - (F^* v)_k (F^* u)_l\big)\circ F^{-1} \frac{\p y_l}{\p x_j} .
$$
If $A$ and $B$ are $3\times 3$ matrices defined as
$$
A_{ij} = (u_i v_j - v_i u_j)\circ F\quad\text{and}\quad B_{kl} = (F^* u)_k (F^* v)_l - (F^* v)_k (F^* u)_l,
$$
then the above identity can be rewritten as
$$
A = (DF)^{-T} B (DF)^{-1}.
$$
Clearly, both $A$ and $B$ are skew symmetric. According to the statement right before \cite[Corollary~3.58]{monk2003finite}, this completes the proof.
\end{proof}

Let $\rho\in C^{0,1}(\R^3;\R)$ be a boundary defining function for $\p\Omega$, i.e. $\Omega = \{x\in\R^3: \rho(x)>0\}$ and $\p\Omega = \{x\in\R^3: \rho(x)=0\}$. Then $\rho\circ F$ is a boundary defining function for $\p\widetilde\Omega$. Recall that the outer unit normals to $\p\Omega$ and $\p\widetilde\Omega$ are defined as
$$
\nu := - \frac{\nabla\rho}{|\nabla\rho|}\Bigg|_{\p\Omega}\quad\text{and}\quad \widetilde\nu := - \frac{\nabla(\rho\circ F)}{|\nabla(\rho\circ F)|}\Bigg|_{\p\widetilde\Omega},
$$
respectively.

\section{Identification of the impedance map $Z_{\sigma,\mu}^\omega$ with a
  pseudodifferential operator}
\label{section::Appendix B}

Let $\Omega\subset\R^3$ be a bounded domain with $C^\infty$
boundary. Here, we assume that $\sigma,\mu\in
C^\infty(\overline\Omega)$ such that $\sigma\ge\sigma_0$ and
$\mu\ge\mu_0$ for some constants $\sigma_0,\mu_0>0$. The aim of this
appendix is to show that the impedance map $Z^\omega_{\sigma,\mu}$ is
a pseudodifferential operator of order $1$ following
\cite{mcdowall1997boundary}. Then we study the connection of
$Z^\omega_{\sigma,\mu}$ to notion of apparent resistivity.

\smallskip

We work in a small neighborhood of a point $p\in \p\Omega$. Without loss of generality, we fix a coordinate system near $p$ such that that $p=0$ and that $\Omega\subset\{x\in\R^3: x_3>0\}$ and $\p\Omega\subset \{x\in\R^3: x_3=0\}$ near $p$. As explained in Section~\ref{section::CGOs}, the equation \eqref{eqn::Maxwell} is equivalent to
\begin{align}
L_{\sigma,\mu}(x,D)H:=-\Delta H - \nabla(\nabla\beta \cdot H) - \nabla\alpha \times \nabla\times H - i\omega \sigma \mu H = 0\quad\text{in}\quad\Omega, \label{LH Appendix B}\\
\nabla\cdot(\mu H) = 0\quad\text{in}\quad \Omega, \label{div free Appendix B}
\end{align}
where $\alpha = \log\sigma$ and $\beta = \log\mu$. We use the notation $D_j=-i\p_j$, $j=1,2,3$,  $x'=(x_1,x_2)$ and $D'=(D_1,D_2)$. Then the equation \eqref{LH Appendix B} can we rewritten as
$$
L_{\sigma,\mu}(x,D)H = (D_3^2 - \Delta' - i M(x,D') - i N(x) D_3 - (\nabla\nabla\beta+i\omega\sigma\mu)) H = 0\quad\text{in}\quad\Omega,
$$
where $\Delta'=\p_1^2+\p_2^2$ and
$$
M = \left(\begin{matrix}
\p_1\beta D_1 - \p_2\alpha D_2 & (\p_2\beta + \p_2\alpha) D_1 & (\p_3\beta + \p_3\alpha) D_1 \\
(\p_1\beta + \p_1\alpha) D_2 & \p_2\beta D_2 - \p_1\alpha D_1 & (\p_3\beta + \p_3\alpha) D_2 \\
0 & 0 & - \p_2\alpha D_2 - \p_1\alpha D_1
\end{matrix}\right),\quad
N = \left(\begin{matrix}
-\p_3\alpha & 0 & 0 \\
0 & -\p_3\alpha & 0 \\
\p_1\beta + \p_1\alpha & \p_2\beta + \p_2\alpha & - \p_3\beta
\end{matrix}\right).
$$
Suppose that $H\in H_{\Div}^1(\Omega)$ solves \eqref{LH Appendix B} and \eqref{div free Appendix B}. Then the impedance map $Z^\omega_{\sigma,\mu}$ near $p$, in the above mentioned coordinates, has the form
\begin{equation}\label{form of the impedance map in Appendix B}
(\nu\times H)|_{\p\Omega} = (H_2, - H_1, 0)|_{x_3=0}\overset{Z^\omega_{\sigma,\mu}}{\mapsto} (\nu\times\sigma^{-1} \nabla\times H)|_{\p\Omega} = \sigma^{-1}(\p_3 H_1 - \p_1 H_3, \p_3 H_2 - \p_2 H_3, 0)|_{x_3=0},
\end{equation}
since $\nu = (0, 0, -1)$. To deal with the appearances of $\p_3 H_j$ and $\p_j H_3$, $j=1,3$, in the above expression, we need the following two results, which can be proven exactly as in Proposition~1 and Proposition~2 of \cite{mcdowall1997boundary}.

\begin{Proposition}\label{lemma b1}
There is a $3\times 3$ matrix-valued pseudodifferential operator $B = B(x,D')$ of order $1$ in $x'$, depending smoothly on $x_3$, such that
$$
L_{\sigma,\mu}(x,D) \equiv (D_3 - iN(x) - iB(x,D'))(D_3 + iB(x,D')).
$$
The principal symbol of $B(x,D')$ is $-|\xi'| I$.
\end{Proposition}
Here and in what follows $\equiv$ denotes an equality modulo a smoothing operator and $\xi_j$ is the dual variable to $D_j$.

\begin{Proposition}\label{lemma b2}
If $H\in H_{\Div}^1(\Omega)$ solves \eqref{LH Appendix B}, then $\p_3 H|_{\p\Omega} \equiv BH|_{\p\Omega}$.
\end{Proposition}
It follows from Proposition~\ref{lemma b2} that
\begin{equation}\label{eqn 3 in Appendix B}
\p_3 H_j \equiv B_{j1}H_1 + B_{j2}H_2 + B_{j3}H_3,\quad j=1,2,3.
\end{equation}
Also, the equation \eqref{div free Appendix B} can be written as
$$
\p_3 H_3 + \p_3\beta H_3 = - (\p_1\beta + \p_1) H_1 - (\p_2\beta + \p_2) H_2.
$$
Combining these two, we obtain
$$
(B_{33} + \p_3\beta)H_3 \equiv - (B_{31} + \p_1\beta + \p_1) H_1 - (B_{32} + \p_2\beta + \p_2) H_2.
$$
We write $J(x,D') = B_{33}(x,D') + \p_3\beta(x)$. Then
\begin{equation}\label{eqn 4 in Appendix B}
JH_3 \equiv - (B_{31} + \p_1\beta + \p_1) H_1 - (B_{32} + \p_2\beta + \p_2) H_2.
\end{equation}
Let $K(x,D')$ be a pseudodifferential operator of order $-1$ in $x'$, depending smoothly on $x_3$, such that $KJ \equiv \id$, where $\id$ is the identity operator. Clearly, the principal symbol  of $J(x,D')$ is $-|\xi'|$. To calculate the principal symbol of $K(x,D')$, we use the identity $KJ \equiv \id$ at the principal symbol level. It follows then from \cite[Corollary~8.38]{folland1995introduction} and \cite[Corollary~8.32]{folland1995introduction} that the principal symbol of $K(x,D')$ is $- |\xi'|^{-1}$.\smallskip

Using \eqref{eqn 3 in Appendix B} and \eqref{eqn 4 in Appendix B}, one can show that $Z_{\sigma,\mu}^\omega$ is a $2\times 2$ matrix-valued pseudodifferential operator of order $1$ such that
\begin{align*}
(Z_{\sigma,\mu}^\omega)_{j1} &\equiv \sigma^{-1} B_{j2} + \sigma^{-1} (\p_j - B_{j3}) \circ K \circ (B_{32} + \p_2\beta + \p_2),\\
(Z_{\sigma,\mu}^\omega)_{j2} &\equiv - \sigma^{-1} B_{j1} - \sigma^{-1} (\p_j - B_{j3}) \circ K \circ (B_{32} + \p_2\beta + \p_2)
\end{align*}
for $j=1,2$. From this, one can show that the principal symbol of $Z_{\sigma,\mu}^\omega$ is
$$
\frac{1}{\sigma(x',0)|\xi'|}\left(\begin{matrix}
\xi_1\xi_2 & \xi_2^2\\
- \xi_1^2 & - \xi_1\xi_2
\end{matrix}\right).
$$
The expression for the impedance map can be used to define apparent resistivity, as in \cite{grandis1999bayesian,thiel2008modelling, weckmann2003images}.

\begin{Remark}\label{remark in Appendix B}{\rm
As one can see, the principal symbol of $Z_{\sigma,\mu}^\omega$ determines $\sigma|_{\p\Omega}$ and all its tangential derivatives on $\p\Omega$. Studying the asymptotic expansion of the full symbol of $Z_{\sigma,\mu}^\omega$ as in \cite{mcdowall1997boundary}, one can determine $\p^\gamma \sigma|_{\p\Omega}$ and $\p^\gamma \mu|_{\p\Omega}$ for all multi-indices $\gamma$. To that end, one needs to perform a more detailed analysis of the asymptotic expansions of the full symbols of $B$, constructed in Proposition~\ref{lemma b1}, and $K$. 
}\end{Remark}

\bibliographystyle{siam}
\bibliography{Bibliography}

\end{document}